\documentclass{amsart}

\usepackage{color} % This package was added to write corrections in red

\usepackage{appendix}

\newtheorem{theorem}{Theorem}[section]
\newtheorem{prop}[theorem]{Proposition}
\newtheorem{proposition}[theorem]{Proposition}
\newtheorem{lemma}[theorem]{Lemma}
\newtheorem{cor}[theorem]{Corollary}

\newtheorem{definition}[theorem]{Definition}
\newtheorem{question}[theorem]{Question}
\newtheorem{remark}[theorem]{Remark}

\newcommand{\bd}{{\partial}}
\newcommand{\reals}{{\bf R}}

\newcommand{\zed}{{\bf Z}}

\newcommand{\TT}{{\bf T}}

\newcommand{\vol}{{\rm vol}}

\newcommand{\sphere}{{\mathcal{S}^2}}

\newcommand{\E}{{\mathbb{E}}}
\newcommand{\EE}{{\mathcal{E}}}
\newcommand{\R}{{\mathbb{R}}}
\newcommand{\LL}{{\mathcal{L}}}
\newcommand{\HH}{{H}}
\newcommand{\Prob}{{\rm{Prob}}}

\newcommand{\sgn}{\operatorname{sgn}}

\begin{document}

\title[Curvature of random metrics]{Scalar curvature and
$Q$-curvature of random metrics.}

\author[Y. Canzani]{Yaiza Canzani}
\address{Department of Mathematics and
Statistics, McGill University, 805 Sherbrooke Str. West, Montr\'eal
QC H3A 2K6, Ca\-na\-da.} \email{canzani@math.mcgill.ca}

\author[D. Jakobson]{Dmitry Jakobson}
\address{Department of Mathematics and
Statistics, McGill University, 805 Sherbrooke Str. West, Montr\'eal
QC H3A 2K6, Ca\-na\-da.} \email{jakobson@math.mcgill.ca}

\author[I. Wigman]{Igor Wigman}
\address{Centre de recherches math\'ematiques (CRM),
Universit\'e de Montr\'eal C.P. 6128, succ. centre-ville Montr\'eal,
Qu\'ebec H3C 3J7, Canada \newline
currently at \newline
Institutionen f\"{o}r Matematik, Kungliga Tekniska h\"{o}gskolan (KTH),
Lindstedtsv\"{a}gen 25, 10044 Stockholm, Sweden}
\email{wigman@kth.se}



\thanks{D.J. was supported by NSERC, FQRNT and Dawson fellowship.
\\ I.W. was supported by a CRM-ISM fellowship, Montr\'eal and
the Knut and Alice Wallenberg Foundation, grant KAW.2005.0098}

\begin{abstract}
We study Gauss curvature for random Riemannian metrics on a compact
surface, lying in a fixed conformal class; our questions are
motivated by comparison geometry. Next, analogous questions are
considered for the scalar curvature in dimension $n>2$, and for the
$Q$-curvature of random Riemannian metrics.
\end{abstract}

\maketitle

\section{Introduction} The goal of the authors in this paper is to
investigate standard questions in {\em comparison geometry} for
random Riemannian metrics lying in the same conformal class.

Random metrics have long been considered in $2$-dimensional
conformal field theory and quantum gravity, random surface models
and other fields. In addition, random metrics are frequently
considered in and cosmology and astrophysics, in the study of
gravitational waves and cosmic microwave background radiation.

Random metrics lying in a fixed conformal class are easiest to treat
analytically; in addition, many classical problems in differential
geometry are naturally formulated and solved for metrics lying in a
fixed conformal class (uniformization theorem for Gauss curvature in
dimension $2$, Yamabe problem and uniformization problem for
$Q$-curvature in higher dimensions).  Accordingly, it is natural to
consider random metrics lying in a fixed conformal class.

The questions considered in this paper are motivated by {\em
comparison geometry}. Since the 19th century, many results have been
established comparing geometric and topological properties of
manifolds where the (sectional or Ricci) curvature is bounded from
above or from below, with similar properties of manifolds of
constant curvature. Examples include Toponogov Theorem (comparing
triangles); sphere theorems of Myers and Berger-Klingenberg for
positively-curved manifolds; volume of the ball comparison theorems
of Gromov and Bishop; splitting theorem of Cheeger and Gromoll;
Gromov's pre-compactness theorem; theorems about geodesic flows and
properties of fundamental group for negatively-curved manifolds; and
numerous other results.

When studying such questions for random Riemannian metrics, the
first natural question is to estimate the {\em probability} of the
metric satisfying certain curvature bounds, in a suitable regime.
The present paper addresses such questions for {\em scalar
curvature}, and also for Branson's {\em $Q$-curvature}.

Let $(M,g)$ be an $n$-dimensional compact manifold, $n\geq 2$.
Recall that the {\em Riemann curvature tensor} is defined by
$R(X,Y)Z:=\nabla_X\nabla_YZ-\nabla_Y\nabla_XZ-\nabla_{[X,Y]}Z$,
where $\nabla$ denotes the Levi-Civita connection. In local
coordinates,
$$
R_{ijkl}:=\langle R(\bd_i,\bd_j)\bd_k,\bd_l\rangle.
$$

The {\em Ricci curvature} of $g$ can be defined in local coordinates
by the formula $R_{jk}=g^{il}R_{ijkl}$. In geodesic normal
coordinates, the volume element takes the form
$dV_g=[1-(1/6)R_{jk}x^jx^k+O(|x|^3)]dV_{Euclidean}$.

The {\em Scalar curvature} can be defined in local coordinates by
the formula $R=g^{ik}R_{ik}$.  Geometrically, $R(x_0)$ determines
the difference of the volume of a small ball of radius $r$ in $M$
(centered at $x_0$) and the Euclidean ball of the same radius: as
$r\to 0$,
$$
{\rm vol}(B_M(x_0,r)) =
{\rm vol}(B_{\reals^n}(r))
\left[1-\frac{R(x_0)r^2}{6(n+2)}+O(r^4)\right].
$$

We study the behavior of scalar curvature for random Riemannian
metrics in a fixed conformal class, where the conformal factor is a
random function possessing certain smoothness. We consider random
metrics that are close to a ``reference'' metric that we denote $g_0$.

The paper addresses two main questions:

\begin{question}\label{question1}
Assuming that the scalar curvature $R_0$ of the reference metric
$g_0$ doesn't vanish, what is the probability that the scalar
curvature of the perturbed metric changes sign?
\end{question}

We remark that in each conformal class, there exists a {\em Yamabe
metric} with constant scalar curvature $R_0(x)\equiv R_0$,
\cite{Yam,Au76,Sch84,Tr}; the sign of $R_0$ is uniquely determined.
Problem (i) can be posed in each conformal class where $R_0\neq 0$
(e.g. in dimension two, for $M\not\cong\TT^2$). Also, it was shown in
\cite{CY,DM,N} that in every conformal class satisfying certain {\em
generic} conditions, there exists a metric $g_0$ with constant
$Q$-curvature, $Q_0(x)\equiv Q_0$. Question \ref{question1} can be
posed in each conformal class where $Q_0\neq 0$.

In dimension 2, this problem is studied in section
\ref{sec:BorellTIS} for a.s. $C^0$ metrics on surfaces $S_\gamma$ of
genus $\gamma\neq 1$ (if the ``reference'' metric has scalar
curvature of constant sign). The probability estimates are greatly
improved in section \ref{sec:S2} for a.s. $C^2$ metrics on
$\sphere$. Problem (i) is addressed for scalar curvature in higher
dimensions in section \ref{sec:d>2:var}; and for $Q$-curvature in
section \ref{Q:sign}. In section \ref{sec:rand:analytic} several
comparison theorems are proved for random real-analytic metrics.

\begin{question}\label{question2}
What is the probability that curvature of the perturbed metric
changes by more than $u$ (where $u$ is a positive real parameter,
subject to some restrictions)?
\end{question}

Question \ref{question2} is studied on surfaces in section
\ref{sec:Linfinity}, and for $Q$-curvature in section
\ref{sec:Q:Linfinity}.

In appendix \ref{sec:yamabe}, we include a short survey of the
results on Yamabe problem, and in appendix \ref{sec:metr-posneg} a
short survey on existence of metrics of positive and negative scalar
curvature in conformal classes.

Finally, in appendix \ref{apx:h attain, E[chi]=Prob} we verify the
assumptions needed to apply the results of R. Adler and J. Taylor to
answer Question (i) on the round metric on $S^2$.

Our techniques are inspired by \cite{AT03, ATT05, AT08, Bleecker}.

%%%%%%%%%

\subsection{Conventions}

Given a random field $F:T\rightarrow\R$ on parameter set $T$ we
define the random variable
$$
\|F\|_{T} :=\sup\limits_{t\in T} F(t).
$$
Note that there is no absolute value in the definition of $\| \cdot
\|_{T}$, so that it is by no means a norm; this is in contrast to $\| \cdot
\|_{\infty}$, which denotes the sup norm.
Let $\Psi(u)$ denote the error function
$$ \Psi(u)=\frac{1}{\sqrt{2\pi}}\int_u^\infty e^{-t^2/2}dt.$$

\subsection{Acknowledgement} The authors would like to thank R.
Adler, P. Guan, V. Jaksic, N. Kamran, S. Molchanov, I. Polterovich,
G. Samorodnitsky, B. Shiffman, J. Taylor, J. Toth, K. Worsley and S.
Zelditch for stimulating discussions about this problem. The authors
would like to thank for their hospitality the organizers of the
following conferences, where part of this research was conducted:
``Random Functions, Random Surfaces and Interfaces'' at CRM
(January, 2009); ``Random Fields and Stochastic Geometry'' at Banff
International Research Station (February, 2009). D.J. would like to
also thank the organizers of the program ``Selected topics in
spectral theory'' at Erwin Shr\"odinger Institute in Vienna (May
2009), as well as the organizers of the conference ``Topological
Complexity of Random Sets'' at American Institute of Mathematics in
Palo Alto (August 2009).

%%%%%%%%

\section{Random metrics in a conformal class} We
consider a conformal class of metrics on a Riemannian manifold $M$
of the form
\begin{equation}\label{conf:change}
g_1=e^{af}g_0,
\end{equation}
where $g_0$ is a ``reference'' Riemannian metric on $M$, $a$ is a
constant, and $f=f(x)$ is a $C^2$ function on $M$.

Given a metric $g_0$ on $M$ and the corresponding Laplacian
$\Delta_0$, let $\{\lambda_j,\phi_j\}$ denote an orthonormal basis
of $L^2(M)$ consisting of eigenfunctions of $-\Delta_0$; we let
$\lambda_0=0,\phi_0=1$. We define a random conformal multiple $f(x)$
by
\begin{equation}\label{rand:conf}
f(x)=-\sum\limits_{j=1}^{\infty}a_j c_j\phi_j(x),
\end{equation}
where $a_j\sim \mathcal{N}(0,1)$ are i.i.d standard Gaussians, and
$c_j$ are positive real numbers, and we use the minus sign for convenience
purposes only. We assume that $c_j=F(\lambda_j)$, where $F(t)$ is an
eventually monotone decreasing function of $t$, $F(t)\to 0$ as
$t\to\infty$. For example, we may take $c_j=e^{-\tau\lambda_j}$ or
$c_j=\lambda_j^{-s}$. Equivalently we equip the space of functions
(distributions) $L^{2}(M)$ with the probability measure $\nu =
\nu_{\{ c_{n} \}_{n=1}^{\infty}}$ generated by the densities on the
finite cylinder sets
\begin{equation}\label{eq:dnu cyl sets gen}
d\nu_{(n_{1},n_{2},\ldots n_{l})}(f) =
\frac{1}{\prod\limits_{j=1}^{l}(2\pi c_{n_{j}}^2)^{1/2}}
\exp\left(-\frac{1}{2} \sum\limits_{i=1}^{l}
\frac{f_{n_{j}}^2}{c_{n_{j}}^2}  \right) df_{n_{1}}\ldots df_{n_{l}},
\end{equation}
where $f_{n} = \langle f, \phi_{j}\rangle _{L^{2}(M)}$ are the Fourier
coefficients.

The {\em random field} $f(x)$ is a centered Gaussian field with covariance
function
\begin{equation*}
r_{f}(x,y) := \E[f(x)f(y)] = \sum\limits_{j=1}^{\infty}
c_{j}^2\phi_{j}(x)\phi_{j}(y),
\end{equation*}
$x,y\in M$.
In particular for every $x\in M$, $f(x)$ is mean zero Gaussian of variance
\begin{equation*}
\sigma^{2}(x) = r_{f}(x,x) = \sum\limits_{j=1}^{\infty} c_{j}^2\phi_{j}(x)^2.
\end{equation*}
For special manifolds such as the $2$-dimensional sphere
$\sphere\subseteq\R^{3}$ it will be convenient to parameterize
$f(x)$ in a different fashion (see \eqref{eq:f sphere def}).

The central object of the present study is the scalar curvature
resulting from the conformal change of the metric
\eqref{conf:change}. For $n=2$, the expression \eqref{eq:R1 after
conf} below for the curvature is of a particular simple shape,
studying which it is convenient to work with the random centered
Gaussian field
\begin{equation}
\label{eq:h=Deltaf gen def} h(x) := \Delta_{0} f(x) =
\sum\limits_{j=1}^{\infty}a_j c_j \lambda_{j} \phi_j(x)
\end{equation}
having the covariance function
\begin{equation}\label{variance:h}
r_{h}(x,y) = \sum\limits_{j=1}^{\infty}c_j^{2} \lambda_{j}^{2} \phi_j(x) \phi_{j}(y)
\end{equation}
$x,y\in M$. In principle, one may derive
any property of $h$ in terms of the function $r_h$ and its
derivatives by the Kolmogorov theorem.

%%%%%%%

\subsection{Smoothness}
We refer to \cite[Ch. 2]{Au98} for definitions and basic facts about
Sobolev spaces.

The smoothness of the Gaussian random field \eqref{rand:conf} is
given by the following proposition, \cite[Proposition 1]{Bleecker}:
\begin{prop}\label{prop:sobolev_regularity}
If $\sum_{j=1}^\infty(\lambda_j+1)^r c_j^2<\infty$, then $f(x)\in
H^r(M)$ a.s. Equivalently, the measure $\nu$ defined as
\eqref{eq:dnu cyl sets gen} is concentrated on $H^r(M)$, i.e.
$\nu(H^r)=1$.
\end{prop}
Choosing $c_j=F(\lambda_j)=\lambda_j^{-s}$ translates to
$\sum_{j\geq 1}\lambda_j^{r-2s}<\infty$.  In dimension $n$, it
follows from Weyl's law that $\lambda_j\asymp j^{2/n}$ as
$j\to\infty$; we find that
$$
If\ s>\frac{2r+n}{4},\ \ \ then\ f(x)\in H^r(M)\  a.s.
$$
By the Sobolev embedding theorem, $H^r\subset C^k$ for $k<r-n/2$.
Substituting into the formula above, we find that
\begin{equation}\label{Ck:as}
If\  c_j=O(\lambda_j^{-s}),s>\frac{n+k}{2}, \ \ \ then\ f(x)\in C^k\
a.s.
\end{equation}

We will be mainly interested in $k=0$ and $k=2$.  Accordingly, we
formulate the following
\begin{cor}\label{C0:C2}
If $c_j=O(\lambda_j^{-s}),s>n/2,$ then $f\in C^0$ a.s; if
$c_j=O(\lambda_j^{-s}),s>n/2+1,$ then $f\in C^2$ a.s.  Similarly, if
$c_j=O(\lambda_j^{-s}),s>n/2+1,$ then $\Delta_0 f\in C^0$ a.s; if
$c_j=O(\lambda_j^{-s}),s>n/2+2,$ then $\Delta_0 f\in C^2$ a.s.
\end{cor}

%%%%%%%

\subsection{Volume}
We next consider the volume of the random metric in
\eqref{conf:change}. The volume element $dV_1$ corresponding
to $g_1$ is given by
\begin{equation}\label{volume:conf}
dV_1=e^{naf/2}dV_0,
\end{equation}
where $dV_0$ denotes the volume element corresponding to $g_0$.

We consider the random variable $V_1=\vol(M,g_1)$.
We shall prove the following
\begin{prop}\label{2dvol:exp}
Notation as above,
$$
\lim_{a\to 0} \E[V_1(a)]=V_0,
$$
where $V_0$ denotes the volume of $(M,g_0)$.
\end{prop}

\begin{proof}
Recall that $f(x)$ defined by \eqref{rand:conf} is a mean zero
Gaussian with variance $\sigma(x)^2 =r_{f}(x,x). $ One may compute
explicitly
\begin{equation*}
\E[e^{naf(x)/2}] = e^{\frac{1}{8}n^2a^2r_{f}(x,x)},
\end{equation*}
so that \eqref{volume:conf} implies that
\begin{equation*}
\E[dV_{1}(x)] = e^{\frac{1}{8}n^2a^2r_{f}(x,x)}dV_{0}(x).
\end{equation*}
Hence, using Fubini we obtain
\begin{equation*}
\E [V_{1}(a)] = \int\limits_{M}\E [dV_{1}] =
\int\limits_{M}e^{\frac{1}{8}n^2a^2r_{f}(x,x)}dV_{0}.
\end{equation*}
Since $r_{f}(x,x)$ is continuous, as $a\rightarrow 0$, the latter
converge to $V_{0}$ by the dominated convergence theorem (say).
\end{proof}

\begin{remark}\label{remark:local-geometry}
The smoothness of the metric $g_1=e^{af}g_0$ is almost surely
determined by the coefficients $c_j$. In some sense, $a$ can be
regarded as the radius of a sphere (in an appropriate space of
Riemannian metrics on $M$) centered at $g_0$. Most of the results in
this paper hold in the limit $a\to 0$; thus, we are studying {\em
local} geometry of the space of Riemannian metrics on $M$.
\end{remark}

%%%%%%%%

\subsection{Scalar curvature in a conformal class} It is well-known
that the scalar curvature $R_1$ of the metric $g_1$ in
\eqref{conf:change} is related to the scalar curvature $R_0$ of the
metric $g_0$ by the following formula (\cite[\S 5.2, p. 146]{Au98})
\begin{equation}
\label{eq:R1 after conf}
R_1  =e^{-af}[R_0-a(n-1)\Delta_0 f-a^2(n-1)(n-2)|\nabla_0 f|^2/4],
\end{equation}
where $\Delta_0$ is the (negative definite) Laplacian for $g_0$, and
$\nabla_0$ is the gradient corresponding to $g_0$. We observe that
the last term vanishes when $n=2$:
\begin{equation}\label{curv2d:conf}
R_1 =e^{-af}[R_0-a\Delta_0 f].
\end{equation}

Substituting \eqref{rand:conf}, we find that
\begin{equation}\label{R1:surface}
R_1(x)e^{af(x)}=R_0(x)-a\sum\limits_{j=1}^{\infty}\lambda_j
a_jc_j\phi_j(x).
\end{equation}

The smoothness of the scalar curvature for the metric $g_1$ is
determined by the random field $a(n-1)\Delta f+a^2(n-1)(n-2)|\nabla
f|^2/4.$

We remark that it follows easily from \eqref{eq:R1 after conf} and
Corollary \ref{C0:C2} that
\begin{prop}\label{R1:smooth}
If $R_0\in C^0$ and $c_j=O(\lambda_j^{-s}),s>n/2+1$ then $R_1\in
C^0$ a.s. If $R_0\in C^2$ and $c_j=O(\lambda_j^{-s}),s>n/2+2$ then
$R_1\in C^2$ a.s.
\end{prop}

Consider the {\em sign} of the scalar curvature $R_1$ of the new
metric.  We make a remark that will be important later:
\begin{remark}\label{sign:remark}
Note that the quantity $e^{-af}$ is positive so that
the sign of $R_1$ satisfies
$$
\sgn(R_1)=\sgn[R_0-a(n-1)\Delta_0 f-a^2(n-1)(n-2)|\nabla_0 f|^2/4],
$$
in particular for $n=2$, assuming that $R_0$ has constant sign, we
find that
$$
\sgn(R_1)=\sgn(R_0-a\Delta_0 f)=
\sgn(R_0-ah)=\sgn(R_0)\cdot\sgn(1-ah/R_0).
$$
\end{remark}

We shall later study similar questions for Branson's $Q$-curvature
\cite{BG,DM,N}.

%%%%%%

\section{Using Borel-TIS inequality to estimate the probability that
$R_1$ changes sign}\label{sec:BorellTIS}

In this section, we shall use Borell-TIS inequality \ref{borelltis
cor} to estimate the probability that curvature of a random metric
on a compact orientable surface $M$ of genus $\gamma\neq 1$ changes
sign. We remark that by Gauss-Bonnet theorem, for $M=\TT^2$ we have
$\int_M R=0$, so the curvature has to change sign on $\TT^2$ (while
for flat metrics, $R_0\equiv 0$).

We denote by $M=M_\gamma$ a compact surface of genus $\gamma\neq 1$.
We choose a reference metric $g_0$ so that $R_0$ has constant sign
(positive iF $M=S^2$, and negative if $M$ has genus $2$).  We remark
that by uniformization theorem, such metrics exist in every
conformal class.  In fact, every metric on $M$ is conformally
equivalent to a metric with $R_0\equiv const$.

Define the random metric on $M_\gamma$ by $g_1=e^{af}g_0$, (as in
\eqref{conf:change}) and $f$ is given by \eqref{rand:conf}, as
usually.

In this section we shall estimate the probability $P_2(a)$ defined
by
\begin{equation}\label{prob:sign-change}
P_2(a):=\Prob\{\exists x\in M:\sgn R_1(g_1(a),x)\neq\sgn(R_0)\},
\end{equation}
i.e. that the curvature $R_1$ of the random metric $g_1(a)$ {\em
changes sign} somewhere on $M$.  The probability of the
complementary event $P_1(a)=1-P_2(a)$ is clearly
$$
P_1(a):=\Prob \{\forall x\in M : \sgn(R_1(g_1(a),x))=\sgn(R_0),\}
$$
i.e. the curvature of the random metric $g_1(a)$ {\em does not}
change sign.

Recall by Remark \ref{sign:remark}, in dimension two
$\sgn(R_1)=\sgn(R_0)\sgn(1-ah/R_0)$, where $h=\Delta_0f$ was defined
earlier in \eqref{eq:h=Deltaf gen def}. We let $v$ denote the random
field
\begin{equation}\label{v:def}
v(x)=h(x)/R_0(x)
\end{equation}

We remark that
\begin{equation}\label{var:variable:curv}
r_v(x,x)=r_h(x,x)/[R_0(x)]^2,
\end{equation}
and we let
\begin{equation}\label{supvar:variable}
\sigma_v^2=\sup_{x\in M} r_v(x,x)=\sup_{x\in M} r_h(x,x)/[R_0(x)]^2.
\end{equation}

We denote by $||v||_M:=\sup_{x\in M} v(x)$. It follows from Remark
\ref{sign:remark} that
\begin{equation}\label{P2:supremum}
P_2(a)=\Prob\left\{||v||_M>1/a\right\}
\end{equation}

We shall estimate $P_2(a)$ in the limit $a\to 0$.  Geometrically,
that means that $g_1(a)\to g_0$, so $P_2(a)$ should go to zero as
$a\to 0$; below, we shall estimate the {\em rate}.  To do that, we
shall use a strong version of the Borell-TIS inequality
(\cite{Borell,TIS}) formulated below.

The proof of the following result can be found in \cite{Borell,TIS},
or in \cite[p. 51]{AT08}
\begin{theorem}[Borel-TIS]
\label{borelltis cor}
Let $f$ be a centered Gaussian process, a.s.
bounded on $M$, and $\sigma_M^2:=\sup_{x\in M} \E[f(x)^2]$. Then
$\E\{||f||_M\} < \infty$, and there exists a constant $\alpha$
depending only on $\E\{||f||_M\}$ so that for $u>E\{||f||_M\}$ we have
$$\Prob\{||f||_M > u\} \leq
e^{\alpha u-u^2/ (2\sigma_M^2)}.$$
\end{theorem}

From now on we shall assume that $R_0\in C^0(M)$, and that
$c_j=O(\lambda_j^{-s}),s>2$. Then Proposition \ref{R1:smooth}
implies that $h$ and $R_1$ are a.s. $C^0$ and hence bounded, since
$M$ is compact.

Recall that $h(x):=\sum_{j=1}^\infty \lambda_jc_ja_j \phi_j(x)$; it
follows that the variance of $v=h/R_0$ is equal to
$(\sum_{j=1}^\infty c_j^2\lambda_j^2 \phi_j(x)^2)/(R_0(x)^2)$.

Recall \eqref{supvar:variable} that $\sigma_v^2=\sup_{x\in
M}r_v(x,x)$; assume that the supremum is attained at $x=x_0$. We
shall use \eqref{P2:supremum} to estimate $P_2(a)$ from above and
below. To get a lower bound for $\Prob\{\|v\|_{M}\} > 1/a)$, choose
$x=x_0$. Clearly,
$$
\Prob\{\|v\|_{M}\} > 1/a\}\geq \Prob\{v(x_0)
> 1/a\}.
$$
The random variable $v(x_0)$ is Gaussian with mean $0$ and variance
$\sigma_v^2$. Accordingly,
\begin{equation}\label{eq:prob_below:NC}
\Prob\{v(x_0)> 1/a\}=\Psi\left(\frac{1}{a\sigma_v}\right),
\end{equation}
where we denote the error function
$$\Psi(u) = \frac{1}{\sqrt{2\pi}}\int\limits_{u}^{\infty} e^{-t^2/2}dt.$$

We obtain an upper bound by a straightforward application of Theorem
\ref{borelltis cor} on our problem.
\begin{proposition}\label{claim:negative1}
There exist a constant $C$ so that
$$\Prob\{\|v\|_{M}) > 1/a\}
 \leq e^{C/a-1/(2a^2\sigma_v^2)}.$$
\end{proposition}

Combining Proposition \ref{claim:negative1} and
\eqref{eq:prob_below:NC} we obtain the following theorem. Note that
the coefficient of $a^{-2}$ in the exponent is the same on both
sides.

\begin{theorem}\label{prop:negative1}
Assume that $R_0\in C^0(M)$ and that $c_j=O(\lambda_j^{-s}),s>2$.
Then there exist constants $C_1>0$ and $C_2$ such that the
probability $P_2(a)$ satisfies
$$
(C_1 a) e^{-1/(2a^2\sigma_v^2)}\leq P_2(a)\leq
e^{C_2/a-1/(2a^2\sigma_v^2)},
$$
as $a\to 0$.  In particular
$$
\lim_{a\to 0} a^2\ln P_2(a)=\frac{-1}{2\sigma_v^2}.
$$
\end{theorem}

\begin{remark}
In section \ref{sec:S2}, we shall greatly improve the result of
Theorem \ref{prop:negative1} and obtain much more precise estimates
of $P_2(a)$ for $M=\sphere$ (see Theorem \ref{thm:prob R1>0 S2}
below) using the results of Adler and Taylor described in the next
section. To apply Borell-TIS inequality, $h$ is required to be
a.s. $C^0$. To apply the results of Adler-Taylor, $h$ needs to be
a.s. $C^2$. We hope to improve the estimates in Theorem
\ref{prop:negative1} in a forthcoming paper.
\end{remark}

%%%%%%%%%%%%%%%

\subsection{Random real-analytic metrics and comparison
results}\label{sec:rand:analytic}

In this section, we let $M$ be a compact orientable surface,
$M\not\cong\TT^2$. We shall consider random {\em real-analytic}
conformal deformations; this corresponds to the case when the
coefficients $c_j$ in \eqref{rand:conf} decay {\em exponentially}.
We shall use standard estimates for the heat kernel to estimate the
probabilities computed in the previous section \ref{sec:BorellTIS}.

We fix a real parameter $T>0$ and choose the coefficients $c_j$ in
\eqref{rand:conf} to be equal to
\begin{equation}\label{coeff:rand-analytic}
c_j=e^{-\lambda_jT/2}/\lambda_j.
\end{equation}

Then it follows from \eqref{variance:h} that
$$
r_h(x,x)=e^*(x,x,T)=\sum_{j:\lambda_j>0} e^{-\lambda_jT}\phi_j(x)^2,
$$
where $e^*(x,x,T)$ denotes the {\em heat kernel} on $M$ {\em without
the constant term}, evaluated at $x$ at time $T$.

The heat kernel $e(x,y,t)=\sum_{j}e^{-\lambda_jt}\phi_j(x)\phi_j(y)$
defines a fundamental solution of the heat equation on $M$. It is
well-known that $e(x,y,t)$ is smooth in $x,y,t$ for $t>0$, and that
$e^*(x,y,t)$ decays exponentially in $t$, \cite{Chavel,Gilkey}.

%%%%%%

\subsection{Comparison Theorem: $T\to 0^+$}

The following asymptotic expansion for the heat kernel is standard
\cite{Gilkey}:
$$
e(x,x,T)\sim_{t\to 0^+}\frac{1}{(4\pi)^{n/2}}\sum_{j=0}^\infty
a_j(x)T^{j-n/2};
$$
here $a_j(x)$ is the $j$-th {\em heat invariant}, where
$$
a_0(x)=1,\ a_1(x)=R(x)/6.
$$
In particular,
$$
\lim_{T\to 0^+} e(x,x,T)T^{n/2}=\frac{1}{(4\pi)^{n/2}}.
$$

Combining with \eqref{supvar:variable}, we obtain the following
\begin{prop}\label{supvar-smallT}
Assume that the coefficients $c_j$ are chosen as in
\eqref{coeff:rand-analytic}.  Then as $T\to 0^+$, $\sigma_v^2$ is
asymptotic to
$$
\frac{1}{(4\pi T)^{n/2}\inf_{x\in M} (R_0(x))^2} .
$$
\end{prop}
That is, as $T\to 0^+$, the probability $P_2(a)$ is determined by
the value of
$$
\inf_{x\in M} (R_0(x))^2.
$$

Proposition \ref{supvar-smallT} is next applied to prove a
comparison theorem.  Let $g_0$ and $g_1$ be two distinct reference
metrics on $M$, normalized to have equal volume, and such that
$R_0\equiv const$ and $R_1\not\equiv const$.

\begin{theorem}\label{thm:an:smallT}
Let $g_0$ and $g_1$ be two distinct reference metrics on $M$,
normalized to have equal volume, such that $R_0$ and $R_1$ have
constant sign, $R_0\equiv const$ and $R_1\not\equiv const$. Then
there exists $a_0,T_0>0$ (that depend on $g_0,g_1$) such that for
any $0<a<a_0$ and for any $0<t<T_0$, we have
$P_2(a,T,g_1)>P_2(a,T,g_0)$.
\end{theorem}

\noindent{\bf Proof:} It follows from Gauss-Bonnet's theorem that
$$
\int_M R_0 dV_0=\int_M R_1 dV_1.
$$
Since ${\rm vol}(M,g_0)={\rm vol}(M,g_1)$, and since by assumption
$R_0\equiv const$ and $R_1\not\equiv const$, it follows that
$$
b_0:=\min_{x\in M} (R_0(x))^2 > \inf_{x\in M} (R_1(x))^2:=b_1.
$$

Accordingly, as $T\to 0^+$, we have
$$
\frac{\sigma_v^2(g_1,T)}{\sigma_v^2(g_0,T)}\asymp \frac{b_0}{b_1}>1.
$$

The result now follows from Theorem \ref{prop:negative1}.

\qed

It follows that in every conformal class, $P_2(a,T,g_0)$ is
minimized in the limit $a\to 0, T\to 0^+$ for the metric $g_0$ of
constant curvature.

%%%%%%

\subsection{Comparison Theorem: $T\to\infty$}

Let $M$ be a compact surface, where the scalar curvature $R_0$ of
the reference metric $g_0$ has constant sign.  Let
$\lambda_1=\lambda_1(g_0)$ denote the smallest nonzero eigenvalue of
$\Delta_0$. Denote by $m=m(\lambda_1)$ the multiplicity of
$\lambda_1$, and let
\begin{equation}\label{sup:eigenspace}
F:=\sup_{x\in M} \frac{\sum_{j=1}^m\phi_j(x)^2}{R_0(x)^2}.
\end{equation}
The number $F$ is finite by compactness and the assumption that
$R_0$ has constant sign on $M$.

\begin{prop}\label{prop:supvar-largeT}
Let the coefficients $c_j$ be as in \eqref{coeff:rand-analytic}.
Denote by $\sigma_v^2(T)$ the corresponding supremum of the variance
of $v$. Then
\begin{equation}\label{supvar:largeT}
\lim_{T\to\infty}\frac{\sigma_v^2(T)}{F e^{-\lambda_1 T}}=1.
\end{equation}
\end{prop}

\noindent{\bf Proof of Proposition \ref{prop:supvar-largeT}:} Recall
that it follows from \eqref{supvar:variable} that
$$
r_v(x,x)=\frac{e^*(x,x,T)}{R_0(x)^2}.
$$
We write $e^*(x,x,T)=e_1(x,T)+e_2(x,T)$, where
$$
e_1(x,T)=e^{-\lambda_1 T}\sum_{j=1}^m \phi_j(x)^2,
$$
and
$$
e_2(x,T)=\sum_{j=m+1}^\infty e^{-\lambda_j T} \phi_j(x)^2.
$$

Clearly, as $T\to\infty$, we have
$$
\lim_{T\to\infty} e^{\lambda_1 T}\sup_{x\in M}
\frac{e_1(x,T)}{R_0(x)^2} =F,
$$
where $F$ was defined in \eqref{sup:eigenspace}. It suffices to show
that as $T\to\infty$,
\begin{equation}\label{e2:small}
\frac{e_2(x,T)}{R_0(x)^2}=o\left(e^{-\lambda_1T}\right)
\end{equation}

Note that by compactness, there exists $C_1>0$ such that
$(1/C_1)\leq R_0^2(x)\leq C_1$ for all $x\in M$.  Accordingly, it
suffices to establish \eqref{e2:small} for $\sup_{x\in M} e_2(x,T)$.

We let $\mu:=\lambda_{m+1}-\lambda_m$; note that
$\lambda_m=\lambda_1$ by the definition of $m$. We have
\begin{equation}\label{heat2:sum}
e_2(x,T)= e^{-\lambda_1 T}\sum_{j=m+1}^\infty
e^{-(\lambda_j-\lambda_1) T} \phi_j(x)^2.
\end{equation}
Let $k$ be the smallest number such that $\lambda_{k}>2\lambda_1$.

We rewrite the sum in \eqref{heat2:sum} as $e_2=e_3+e_4$, where the
first term $e_3$ is given by
\begin{equation}\label{e3}
e_3(x,T):=\sum_{j=m+1}^{k-1} e^{-(\lambda_j-\lambda_1) T}
\phi_j(x)^2\leq e^{-\mu T}\sup_{x\in M}\sum_{j=m+1}^{k-1}
\phi_j(x)^2,
\end{equation}
where the last supremum (which we denote by $H$) is finite by
compactness.

The second term $e_4$ is given by
\begin{equation}\label{e4}
e_4(x,T):= \sum_{j=k}^\infty e^{-(\lambda_j-\lambda_1) T}
\phi_j(x)^2\leq \sum_{j=k}^\infty e^{-\lambda_jT/2}  \phi_j(x)^2\leq
\sup_{x\in M}e^*(x,x,T/2).
\end{equation}
We remark that as $T\to\infty$, $e^*(x,x,T/2)\to 0$ exponentially
fast, uniformly in $x$.

Combining \eqref{e3} and \eqref{e4}, we find that
$$
e_2(x,T)=O\left(\left[H\cdot e^{-\mu T}+e^*(x,x,T/2)\right]
e^{-\lambda_1 T}\right),
$$
establishing \eqref{e2:small} for $\sup_{x\in M} e_2(x,T)$ and
finishing the proof of Proposition \ref{prop:supvar-largeT}.

\qed

\begin{theorem}\label{thm:an:largeT}
Let $g_0$ and $g_1$ be two reference metrics (of equal area) on a
compact surface $M$,  such that $R_0$ and $R_1$ have constant sign,
and such that $\lambda_1(g_0)>\lambda_1(g_1)$.  Then there exist
$a_0>0$ and $0<T_0<\infty$ (that depend on $g_0,g_1$), such that for
all $a<a_0$ and $T>T_0$ we have $P_2(a,T;g_0)<P_2(a,T;g_1)$.
\end{theorem}

\noindent{\bf Proof of Theorem \ref{thm:an:largeT}:} By Proposition
\ref{prop:supvar-largeT}, we find that for $T>T_1=T_1(g_0,g_1)$
there exists $C>0$ such that
$$
\frac{1}{C}\leq\frac{\sigma_v^2(T,g_1)e^{\lambda_1(g_1)
T}}{\sigma_v^2(T,g_2)e^{\lambda_1(g_2) T}}\leq C
$$

Accordingly, if we choose $T_2$ so that
$e^{(\lambda_1(g_1)-\lambda_1(g_2))T_2}>C$, and take
$T>\max\{T_1,T_2\}$, we find that Theorem \ref{thm:an:largeT}
follows from the formula above and Theorem \ref{prop:negative1}.

\qed

It was proved by Hersch in \cite{Hersch} that for $M=S^2$, if we
denote by $g_0$ the round metric on $S^2$, then
$\lambda_1(g_0)>\lambda_1(g_1)$ for any other metric $g_1$ on $S^2$
of equal area.  This immediately implies the following
\begin{cor}\label{S2:largeT}
Let $g_0$ be the round metric on $S^2$, and let $g_1$ be any other
metric of equal area. Then, there exist $a_0>0$ and $T_0>0$
(depending on $g_1$) such that for all $a<a_0$ and $T>T_0$ we have
$P_2(a,T;g_0)<P_2(a,T;g_1)$.
\end{cor}

It seems interesting to establish comparison results for finite
times $0<T<\infty$. In fact, it was proved in \cite{Morpurgo} that
the heat trace for the round metric on $S^2$ locally minimizes the
heat trace for all metrics on $S^2$ of the same volume, in an
$L^\infty$ neighborhood of the set of conformal factors on $S^2$;
the size of the neighborhood depends on the interval $[a,b]\subset
(0,\infty)$, where $T\in[a,b]$.  It was also shown in \cite{EI},
that the round metric on $S^2$ was the unique critical metric on
$S^2$ for heat trace functional.  Accordingly, it seems natural to
conjecture that the round metric on $S^2$ will be extremal for
$P_2(a,T)$ for all $T$, in the limit $a\to 0$.

For surfaces of genus $\gamma\geq 2$ the situation is different. The
following result was proved in \cite[Theorem 2.3]{Bryant} (D.J.
first learned about it from S. Wolpert); the corresponding result
for minimal surfaces in $\reals^n$ was established in \cite[Thm.
6]{Yau74}.
\begin{prop}\label{hyperb:notsup}
Let $g_0$ be a hyperbolic metric on a compact orientable surface $M$
of genus $\gamma\geq 2$.  Then $g_0$ {\em does not} maximize
$\lambda_1$ in its conformal class.
\end{prop}

It is well-known that a metric $g_0$ that is extremal for
$\lambda_k$ among all metrics of the same volume in the same
conformal class admits a {\em minimal immersion into $S^m$} by
eigenfunctions that form an orthonormal basis in the eigenspace
$E(\lambda_k)$, where $\dim E(\lambda_k)=m+1$, cf.
\cite{EI03,EI08,Nadirashvili,Takahashi}.  Strong results about the
existence of such metrics were established recently in \cite{NS}.

The metrics that are extremal for $\lambda_k$ among {\em all}
metrics of the same volume (and not just in the same conformal
class) admit {\em isometric} minimal immersions into round spheres
by the corresponding eigenfunctions, see the references above, as
well as \cite{Hersch,LY,YY}. A metric that maximizes $\lambda_1$ for
surfaces of genus $2$ is a branched covering of the round
$2$-sphere, cf. \cite{JLNNP}.

Accordingly, we conclude that on surfaces of genus $\gamma\geq 2$,
different metrics maximize $P_2(a,T)$ in the limit $a\to 0,T\to 0$
and in the limit $a\to 0,T\to\infty$, unlike the situation on $S^2$.

%%%%%%%%%%%%%%%%%%%%%%%%%%

\section{Using results of \cite{AT08} on the $2$-Sphere}\label{sec:S2}

The sphere is special in that the curvature perturbation is {\em isotropic},
so that in particular the variance is constant. In this case a
special theorem due
to Adler-Taylor gives a precise asymptotics for the excursion probability.

\subsection{Random function in Sobolev spaces}

For an integer $m$ let $\EE_{m}$ be the space of spherical harmonics
of degree $m$ of dimension $N_{m} = 2m+1$ associated to the
eigenvalue $E_{m} = m(m+1)$, and for every $m$ fix an $L^{2}$
orthonormal basis  $B_{m}=\{\eta_{m,k} \}_{k=1}^{N_{m}}$ of
$\EE_{m}$.

To treat the spectrum degeneracy it will be convenient to use a
slightly different parametrization of the conformal factor than the
usual one \eqref{rand:conf}
\begin{equation}\label{eq:f sphere def}
f(x) = -\sqrt{|\sphere | }  \sum\limits_{m\ge 1 ,\, k}
\frac{\sqrt{c_{m}}}{E_{m}\sqrt{N_{m}}} a_{m,k} \eta_{m,k}(x),
\end{equation}
where $a_{m,k}$ are standard Gaussian i.i.d. and $c_{m} > 0$ are
some (suitably decaying) constants. For extra convenience we will
assume in addition that
\begin{equation}
\label{eq:cn var 1} \sum\limits_{m=1}^{\infty} c_{m} = 1,
\end{equation}
which has an advantage that a random field $h(x)$ defined
below is of unit variance.

\begin{remark}
We stress once again that for convenience, in the present section,
the random fields $f$ and $h$ (see below) are defined differently
than in the rest of the paper.  The reason for the new definitions
is spectral degeneracy on $S^2$.
\end{remark}

The measure $\nu = \nu_{\{ c_{m} \}_{m=1}^{\infty}}$ corresponding
to \eqref{eq:dnu cyl sets gen} is generated by the densities on the
finite cylinder sets
\begin{equation*}
d\nu_{(m_{1},k_{1}),\ldots (m_{l},k_{l})}(f) =
\frac{1}{\prod\limits_{i=1}^{l}(2\pi s_{m_{i}})^{1/2}}
\exp\left(-\frac{1}{2} \sum\limits_{i=1}^{l}
\frac{f_{(m_{i},k_{i})}^2}{s_{m_{i}}}  \right)
df_{(m_{1},k_{1})}\ldots df_{(m_{l},k_{l})},
\end{equation*}
where $f_{(m,k)} = \langle f,\eta_{m,k}\rangle _{L^{2}(\sphere)}$
are the Fourier coefficients, and
\begin{equation*}
s_{m} := |\sphere | \frac{c_{m}}{E_{m}^2 N_{m}}.
\end{equation*}
Note that $\nu$ is invariant w.r.t. the choice of the orthonormal
basis $\{B_{m}\}_{m=1}^{\infty}$ of the spaces $\EE_{m} $ of the
spherical harmonics, by the invariance of the Gaussian.

Recall that the Sobolev space $\HH_{r}(\sphere)$ consists of
functions (distributions) $g:\sphere\rightarrow\R$, so that
\begin{equation*}
\sum\limits_{m,k} (E_{m}+1)^{r} g_{m,k}^2 < \infty.
\end{equation*}
For example $L^{2}(\sphere) = \HH_{0}(\sphere)$.

By ~\cite{Bleecker}, Proposition 1, $\nu$ is concentrated on
$\HH_{r}(\sphere)$,  if and only if
\begin{equation*}
\sum\limits_{m=1}^{\infty} N_{m} (E_{m}+1)^{r} \frac{c_{m}}{E_{m}^2
N_{m}} < \infty.
\end{equation*}
Since
\begin{equation*}
(E_{m}+1)^{r} \frac{c_{m}}{E_{m}^2 } \asymp m^{2r-4} c_{m},
\end{equation*}
we have the following lemma.

\begin{lemma}\label{Continuity on S2}
Given a sequence $c_{m}$ satisfying \eqref{eq:cn var 1}, we have
$f(x)\in \HH_{r}(\sphere)$ a.s. (or equivalently, the measure $\nu$
defined above satisfies $\nu(\HH_{r})=1$) if and only if
\begin{equation*}
\sum\limits_{m=1}^{\infty} m^{2r-4} c_{m} < \infty.
\end{equation*}
\end{lemma}

In what follows we will always assume that
\begin{equation}
\label{eq:cn <<>>n^-s} c_{m} = O\left(\frac{1}{m^s}\right).
\end{equation}
Thus $f(x)\in \HH_{r}(\sphere)$ precisely for $r <
\frac{s}{2}+\frac{3}{2}$.  Note that if $c_{m} = \frac{K}{m^{s}}$,
\eqref{eq:cn var 1} requires $K = \frac{1}{\zeta(s)}$ where
$\zeta(s)$ is the Riemann zeta function.

%%%%%%%%%

\subsection{Curvature and statement of the main result}

\begin{theorem}
\label{thm:prob R1>0 S2} Let $s>7$, and the metric $g_{1}$ on
$\sphere$ be given by $$g_{1}= e^{af}g_{0} $$ where $f$ is given by
\eqref{eq:f sphere def}. Also, let $c_m\neq 0$ for at least one {\em
odd} $m$. Then as $a\rightarrow 0$, the probability that the
curvature is everywhere positive is given by
\begin{equation*}
\begin{split}
\Prob\{ \forall x\in \sphere.\: R_{1}(x)>0\} &= 1-C_{1} \Psi\left(
\frac{1}{a}  \right)-\frac{C_{2}}{a}\exp\left(
-\frac{1}{2a^2}\right) +o\left(\exp(-\frac{\alpha}{2a^2})\right)
\\&\sim 1-C_{1}\frac{1}{\sqrt{2\pi}}\exp\left( -\frac{1}{2a^2}
\right)-\frac{C_{2}}{a}\exp\left( -\frac{1}{2a^2}\right) ,
\end{split}
\end{equation*}
where $C_1=2$, $C_2= \frac{1}{\sqrt{2\pi}}\sum_{m \geq 1}c_mE_m$ and
$\alpha>1$.
\end{theorem}

The curvature  corresponding to the random Riemannian metric $g_{1} =
e^{af}g_{0}$ is given by
\begin{equation}
\label{eq:R1=1-aDeltaf sphere}
R_1e^{af}=R_0-a\Delta f = 1-a\Delta f,
\end{equation}
where $R_0 \equiv 1$ corresponds to the round metric $g_{0}$. To
make sense of it we will have to assume that $f\in C^{2}(\sphere)$
a.s., for which we will need that $s>3$ (cf. \eqref{eq:cn <<>>n^-s})

It is then natural to introduce the Gaussian random field (cf.
\eqref{eq:h=Deltaf gen def})
\begin{equation}
\label{eq:h=Deltaf def} h(x) := \Delta f(x) = \sqrt{|\sphere | }
\sum\limits_{m\ge 1 ,\, k}  \frac{\sqrt{c_{m}}}{\sqrt{N_{m}}}
a_{m,k} \eta_{m,k}(x),
\end{equation}
so that \eqref{eq:R1=1-aDeltaf sphere} is
\begin{equation}
\label{eq:R1=1-ah sphere}
R_1e^{af}=1-ah.
\end{equation}
The random field $h$ is centered unit variance (see \eqref{eq:cn var
1}) Gaussian isotropic with covariance function
$r_{h}:\sphere\times\sphere\rightarrow\R$ given explicitly by
\begin{equation}
\label{eq:sph cov fnc} r_{h}(x,y):=\E[h(x)h(y) ] =
\sum\limits_{m=1}^{\infty} c_{m} P_{m}(\cos(d(x,y))),
\end{equation}
where $P_{m}$ is the Legendre polynomial of degree $m$ and $d(x,y)$
is the (spherical) distance between $x$ and $y$.

The following lemma follows easily from \eqref{eq:h=Deltaf def},
Proposition \ref{prop:sobolev_regularity} and Sobolev embedding
theorem:
\begin{lemma}\label{R1sphere:smooth}
If $c_m=O(m^{-s}),s>2k+3$, then $h$, and hence $R_1$, are a.s.
$C^k$.
\end{lemma}

The condition \eqref{eq:cn <<>>n^-s} in particular ensures that the
series in \eqref{eq:h=Deltaf def} is a.s. pointwise convergent; we
will need stronger conditions to work with {\em smooth} sample functions.
It then follows from \eqref{eq:R1=1-ah sphere} that
$R_{1}$ is everywhere positive if and only if
\begin{equation*}
\| h \|_{\sphere}:= \sup \limits_{x\in \sphere} \{h(x)\} < \frac{1}{a}.
\end{equation*}

The problem of approximating the excursion probability of
$$E=E_{h,u}:=\left\{\|  h \|_{\sphere} > u := \frac{1}{a}\right\}$$
(i.e., of the complement event) for a given random field $h$ as
$u\rightarrow\infty$ (i.e. $a\rightarrow 0$, small perturbation) is
a classical problem in probability. For the constant variance random
fields (which follows from the isotropic property of $h$), there is
a special precise result due to Adler-Taylor ~\cite{AT03}. The
latter relates $\Prob(E)$ to the expected value of Euler
characteristic of the excursion set $h^{-1}([u,\infty])$, giving an
explicit expression for the latter, where the answer depends on the
Adler-Taylor metric associated to $h$ defined below.

%%%%%%%%%%%%

\subsection{The Adler-Taylor metric of $h$}
Let $h$ be an a.s. $C^1$ random field on a manifold $M$.
The {\em Adler-Taylor} Riemannian metric $g^{AT}$ on $M$ is defined
as follows (cf. \cite[(12.2.2)]{AT08}). Let $x\in M$ and $X,Y\in
T_{x}M$; then
\begin{equation*}
g^{AT}_{h;x}(X,Y) :=\E[X h \cdot Y h].
\end{equation*}
This is the pullback by $x\rightarrow h(x)$ of the standard
structure on $L^{2}$. One may compute $g^{AT}$ in terms of the covariance function as
(\cite{AT08}, p. 306)
\begin{equation*}
g^{AT}_{h;x}(X,Y) = X Y r_{h}(x,y)|_{x=y},
\end{equation*}
Below, we shall specialize to the case $M=\sphere$.

Plugging in \eqref{eq:sph cov fnc} we obtain an expression for the
metric
\begin{equation*}
g^{AT}_{h; x}(X,Y) = \sum\limits_{m=1}^{\infty} c_{m}  \left( X Y
P_{m}(\cos(d(x,y))) \right)|_{x=y},
\end{equation*}
which was computed explicitly for the sphere in ~\cite{W,W1} to be
given by the scalar matrix
\begin{equation}
\label{eq:gAT_h on S2}
g^{AT}_{h;x}= C I_{2}
\end{equation}
with $C=C_{c_{j}} :=  \frac{1}{2}\sum_{m =1}^{\infty} c_m E_m  $, in
any orthonormal basis of $T_{x}(\sphere)$.

\begin{definition}
\label{def:attainable} Let $M$ be a smooth manifold and
$f:M\rightarrow\R$ a smooth centered Gaussian random field with
covariance function $r_{f}(x,y)$. We say that $f$ is {\em
attainable}, if there exists a countable atlas $\mathcal{A} =
(U_{\alpha}, \psi_{\alpha})_{\alpha\in I}$ on $M$, such that for
every $\alpha\in I$, $f^{\alpha} := f \circ \psi_{\alpha}^{-1}$
defined on $\psi(U_{\alpha})\subseteq \R^{2}$,  satisfies:

\begin{enumerate}

\item
\label{it:hi, hij nondeg}

For each $t\in \psi(U_{\alpha})$, the joint distributions of
$$(f^{\alpha}_{i}(t), f^{\alpha}_{ij}(t))_{i<j} \in \R^{5}$$ are nondegenerate,
where $f^{\alpha}_{i}$ and $f^{\alpha}_{ij}$ are the corresponding
partial derivatives of
$f^{\alpha}$ of first and second order respectively.

\item
\label{it:rij reg}

We have (cf. ~\cite{AT08}, (11.3.1))
\begin{equation*}
\max_{i,j} \left|
r_{h_{ij}}(t,t)+r_{h_{ij}}(s,s)-2r_{h_{ij}}(s,t)\right| \le K_{\alpha}
[\ln|t-s|]^{-(1+\beta)}
\end{equation*}
for some $\beta > 0$.

\end{enumerate}

\end{definition}

\begin{theorem}[~\cite{AT08}, Theorem 12.4.1, $2$-dimensional case]
\label{thm:AT E[chi]} Let $f:M\rightarrow\R$ be a
centered, unit variance Gaussian field on a $C^{2}$,
$2$-dimensional manifold $M$. Then
if $f$ is attainable,
\begin{equation*}
\E[\chi(u,+\infty)] = \sum\limits_{j=0}^{2}\mathcal{L}_{j}(M)\rho_{j}(u),
\end{equation*}
where $\LL_{j}(\sphere)$, $j=0,1,2,$ are the Lipschitz-Killing
curvatures of $M$ computed w.r.t. the Adler-Taylor metric
$g^{AT}_{f,\cdot}$ and
\begin{equation*}
\rho_{j}(u) =
\begin{cases}
 \Psi(u)   \quad &j=0, \\
\frac{1}{(2\pi)^{1/2}}e^{-u^2/2} &j=1,\\
\frac{1}{(2\pi)^{3/2}}ue^{-u^2/2} &j=2.
\end{cases}
\end{equation*}
\end{theorem}

We will apply Theorem \ref{thm:AT E[chi]} on the random field
$h=\Delta f$ in the course of proof of the main result of this
section, namely Theorem \ref{thm:prob R1>0 S2}; we relegate the
justification of $h$ being attainable to Appendix \ref{apx:suit reg
S2}.

Theorem \ref{thm:AT E[chi]} allows us to compute the expected Euler
characteristic of the excursion set, which is intimately related to
the excursion probability by ~\cite{AT08}, (14.0.2):
\begin{equation}
\label{eq:Prob ||h||>=u = E[chi]}
\left| \Prob\{\| h \| _{\sphere} \ge u\} - \E[ \chi  \left(
h^{-1}[u,\infty)\right) ]  \right| = O\left( e^{-\alpha u^2 /
2}\right)
\end{equation}
for some $\alpha >1$; this is a meta-theorem  (here we used
assumption \eqref{eq:cn var 1}; otherwise we need to modify
accordingly). Theorem 14.3.3 in ~\cite{AT08} gives sufficient
conditions for this assertions to hold; we relegate validating its
hypotheses to Appendix \ref{apx:Echi f>=u = P f>=0}.

%%%%%%%%%%%%

\subsection{Proof of Theorem \ref{thm:prob R1>0 S2}}

\begin{proof}[Proof of Theorem \ref{thm:prob R1>0 S2}]

We are interested in the probability that for every $x\in\sphere$
\begin{equation*}
h(x) \le u:=\frac{1}{a}
\end{equation*}
for small $a>0$, or, equivalently, its complement
\begin{equation*}
\Prob\{\| h \| _{\sphere} \ge u \}.
\end{equation*}
We employ Theorem \ref{thm:AT E[chi]} due to Adler-Taylor
to compute the expected value of Euler characteristic of the
excursion set explicitly as
\begin{equation}
\label{eq:E[chi]=sum rho_j}
\E[ \chi  \left( h^{-1}[u,\infty)\right) ] =
\sum\limits_{j=0}^{2}\LL_{j}(\sphere)  \rho_{j}(u),
\end{equation}

The statement of Theorem \ref{thm:prob R1>0 S2} then follows from
\eqref{eq:Prob ||h||>=u = E[chi]}, \eqref{eq:E[chi]=sum rho_j}, and
the values of the Lipschitz-killing curvatures of the sphere
relatively to the Adler-Taylor metric \eqref{eq:gAT_h on S2}
$\mathcal{L}_0(\sphere,g^{AT}_{h}) =2$,
$\mathcal{L}_1(\sphere,g^{AT}_{h})=0
$ and $\mathcal{L}_2(\sphere,g^{AT}_{h})=2\pi \left( \sum_{m \geq
1} \frac{c_mE_m}{2}\right)$ (see
~\cite{AT08}, (6.3.8)).
Note that to justify the application of Theorem \ref{thm:prob R1>0 S2}
and \eqref{eq:Prob ||h||>=u = E[chi]} we have to validate the hypotheses
of the corresponding theorems. We do so in Appendix
\ref{apx:h attain, E[chi]=Prob}.

\end{proof}

%%%%%%%%

\section{$L^\infty$ curvature bounds}\label{sec:Linfinity}

\subsection{Definitions and the main result}
In sections \ref{sec:BorellTIS} and \ref{sec:S2} we studied the
probability of the curvature {\em changing sign} after a small
conformal perturbation, on $S^2$ and on surfaces of genus greater
than one.  On the torus $\TT^2$, however, Gauss-Bonnet theorem
implies that the curvature has to change sign for every metric, so
that question is meaningless.

Accordingly, on $\TT^2$ we investigate the probability of another
event that is considered very frequently in comparison geometry: the
probability that scalar curvature satisfies the $L^\infty$ curvature
bounds $||R_1||_\infty<u$, where $u>0$ is a parameter.  Metrics
satisfying such bounds for fixed $u$ are called {\em metrics of
bounded geometry}. The argument on $\TT^2$ is then modified to study
the following natural analogue of the problem on $\TT^2$: estimating
the probability that $||R_1-R_0||_\infty<u$. That question is
considered on $S^2$, and on surfaces of genus greater than one.

In this section we {\em do not} assume that $R_0\equiv const$; nor
do we assume that $R_0$ has constant sign.

\begin{definition}\label{def:3fields}
We shall consider the following three centered random fields on the
surface $\Sigma$:
\begin{itemize}
\item[i)] The random conformal multiple $f(x)$ given by
\eqref{rand:conf}.  We denote its covariance function by $r_f(x,y)$,
and we define $\sigma_f^2=\sup_{x\in\Sigma} r_f(x,x)$.
\item[ii)] The random field $h=\Delta_0 f$ defined in
\eqref{eq:h=Deltaf gen def}.  We denote its covariance function by
$r_h(x,y)$, and we define $\sigma_h^2=\sup_{x\in\Sigma} r_h(x,x)$.
\item[iii)] The random field $w=\Delta_0 f +R_0f=h+R_0f$.  We denote
its covariance function by
$r_w(x,y)$, and we define $\sigma_w^2=\sup_{x\in\Sigma} r_w(x,x)$.
Note that on flat $\TT^2$, $R_0\equiv 0$ and therefore $h\equiv w$.
\end{itemize}
The random fields $f,h$ and $w$ have constant variance on round
$S^2$; also $f$ and $h=w$ have constant variance on flat $\TT^2$.
\end{definition}

We shall prove the following theorem:
\begin{theorem}\label{thm:supbound-general}
Assume that the random metric is chosen so that the random fields
$f,h,w$ are a.s. $C^0$. Let $a\rightarrow 0$ and $u\rightarrow 0$ so
that
\begin{equation}
\label{eq:u/a->inf}
\frac{u}{a}\rightarrow \infty.
\end{equation}
Then
\begin{equation}\label{linfbound:general}
\log \Prob\{\| R_1-R_0  \|_{\infty} > u \} \sim
-\frac{u^2}{2a^2\sigma_w^2}.
\end{equation}
\end{theorem}

\begin{remark}
On the flat $2$-torus, $\sigma_h=\sigma_w$.
\end{remark}

%%%%%%%

\subsection{Proof of Theorem \ref{thm:supbound-general}}

\begin{proof}

In the proof, we shall use (Borel-TIS) Theorem \ref{borelltis cor};
we require that the random fields $f,h,w$ are a.s. $C^0$. Assuming
that $R_0\in C^0(\Sigma)$, sufficient conditions for that are
formulated in Corollary \ref{C0:C2}; we remark that if $h$ is a.s.
$C^0$, then so is $w$.

Let $\Sigma$ will denote a compact orientable surface ($\sphere,\TT^2$
or of genus $\gamma\geq 2$) where the random fields are defined.

\noindent{\bf Step 1.}

We introduce a (large) parameter $S$ that
will be chosen later. On $\TT^2$, we let $B_S$ denote the
``bad'' event where $f$ is {\em large}
\begin{equation}\label{f:large}
B_{S} = \{||f||_\infty>S\}.
\end{equation}
Applying Theorem \ref{borelltis cor}, we find that there exists a
constant $\alpha_f$ such that the following estimate holds:
\begin{equation}\label{BS:est:torus}
{\rm Prob} (B_S) = O\left(
\exp\left(\alpha_fS-\frac{S^2}{2\sigma_f^2} \right)\right).
\end{equation}

On $S^2$ and on surfaces of genus $\geq 2$ we modify the definition
slightly, and let $B_S$ denote the ``bad" event that either $f$ or
$h$ is large
\begin{equation}\label{f,h:large}
B_{S}=\{||f||_\infty>S\}\cup\{ ||h||_\infty>S\}.
\end{equation}

By Theorem \ref{borelltis cor} we find that there exist two
constants $\alpha_f$ and $\alpha_h$ such that
\begin{equation}\label{BS:est:general}
{\rm Prob} (B_S) = O\left(\exp\left(\alpha_f
S-\frac{S^2}{2\sigma_f^2}\right)+ \exp\left(\alpha_h
S-\frac{S^2}{2\sigma_h^2}\right)\right).
\end{equation}

We denote $A_{u,a}$ the event $\{|| R_{1}-R_0 ||_{\infty}
> u\}$; clearly,
\begin{equation}\label{probA:sum}
{\rm Prob} (A_{u,a})={\rm Prob} (A_{u,a}\cap B_S)+{\rm Prob}
(A_{u,a}\cap B_S^c).
\end{equation}
We will choose $S$ later so that
\begin{equation}\label{prob:neglect:BS}
{\rm Prob} (A_{u,a}\cap B_S) = o\left({\rm Prob} (A_{u,a}\cap B_S^c)\right);
\end{equation}
this is only possible
under the assumption \eqref{eq:u/a->inf} of
the present theorem. The inequality \eqref{prob:neglect:BS} implies that
it will be sufficient to
evaluate ${\rm Prob} (A_{u,a}\cap B_S^c)$.

Consider first the event $A_{u,a}\cap B_S^c$. The expression
$R_1-R_0$ is given by \eqref{R1-R0}. We estimate ${\rm Prob}
(A_{u,a}\cap B_S)$ trivially, $${\rm Prob} (A_{u,a}\cap B_S)\leq
{\rm Prob} (B_S).$$  Accordingly, it follows from
\eqref{BS:est:torus} for the torus,
and \eqref{BS:est:general} for the sphere or a surface of genus
$\geq 2$ that in each of the cases
\begin{equation}\label{prob:neglect:A}
{\rm Prob} (A_{u,a}\cap B_S)=
\begin{cases}
O\left(S\cdot \exp\left(\alpha_f
S-\frac{S^2}{2\sigma_f^2}\right)\right), \qquad
&\Sigma=\TT^2;\\
O\left(\exp\left(\alpha_f
S-\frac{S^2}{2\sigma_f^2}\right)+ \exp\left(\alpha_h
S-\frac{S^2}{2\sigma_h^2}\right)\right), \qquad &\text{otherwise}
\end{cases}.
\end{equation}

\noindent{\bf Step 2.}

We next estimate ${\rm Prob} (A_{u,a}\cap B_S^c)$. Recall that in
dimension two, it follows from \eqref{curv2d:conf} that
\begin{equation}\label{R1-R0}
R_1-R_0=R_0(e^{-af}-1)-ae^{-af}\Delta_0 f =
R_0(e^{-af}-1)-ae^{-af}h.
\end{equation}
Note that on $\TT^2$ we have $R_0=0$, and the first term on the
right vanishes, hence we get
$$
R_1=-ae^{-af}h
$$
in that case.

%%%%%%%

\noindent{\bf Step 2a.}

We start with the case $\Sigma=\TT^2$.
We will choose a constant $S$ satisfying
\begin{equation}
\label{eq:aS=o(1)}
aS = o(1).
\end{equation}

On $B_S^c$, we have $|f(x)|=O(S)$, hence $e^{-af(x)} = 1+O(aS)$, so
that
\begin{equation}
\label{eq:Prob Aua BSc=Prob ||Df||>u/a(1+eps)}
\begin{split}
\Prob (A_{u,a}\cap B_{S}^{c}) &= \Prob \left(\left\{ \|h \|_{\infty}
>  \frac{u}{a(1+O(aS))}  \right \}\cap B_{S}^{c}\right)
\\&= \Prob \left(\left\{ \|h \|_{\infty} > \frac{u}{a(1+O(aS))}  \right \}\right) +
O\left( \Prob(B_{S}) \right),
\end{split}
\end{equation}
the last summand being already estimated in \eqref{prob:neglect:A}.
By \eqref{eq:aS=o(1)}, we have $\frac{u}{a(1+O(aS))} \sim
\frac{u}{a}$. Plugging \eqref{prob:neglect:A} and \eqref{eq:Prob Aua
BSc=Prob ||Df||>u/a(1+eps)} into \eqref{probA:sum} we obtain
\begin{equation}
\label{eq:ProbAua=Prob(h>u/a)+O(Se^{-S2})} \Prob (A_{u,a}) = \Prob
\left(\left\{ \|h \|_{\infty} > \frac{u}{a(1+O(aS))} \right
\}\right) + O\left(\exp\left(\alpha_f
S-\frac{S^2}{2\sigma_f^2}\right) \right).
\end{equation}

It then remains to evaluate $$\Prob \left(\left\{ \|h \|_{\infty} >
\frac{u}{a(1+O(aS))}  \right \}\right),$$ and choose $S$ so that the
other term is negligible. To this end we note that by symmetry,
\begin{equation}
\label{eq:probT<prob||infty<2probT}
\begin{split}
\Prob \left(\left\{ \|h \|_{\TT^2} > \frac{u}{a(1+O(aS))}  \right
\}\right) &\le \Prob \left(\left\{ \|h \|_{\infty} >
\frac{u}{a(1+O(aS))}  \right \}\right) \\&\le 2\Prob \left(\left\{
\|h \|_{\TT^2} > \frac{u}{a(1+O(aS))}  \right \}\right),
\end{split}
\end{equation}
and the factor $2$ is negligible on the {\em logarithmic} scale.

To evaluate
\begin{equation}\label{prob:u/a:T2}
\Prob \left(\left\{ \|h \|_{\TT^2} >
\frac{u}{a(1+O(aS))}  \right \}\right)
\end{equation}
we note that \eqref{eq:u/a->inf} together with \eqref{eq:aS=o(1)}
imply that
\begin{equation}
\label{eq:u/a(1+aS)->inf}
\frac{u}{a(1+O(aS))}\rightarrow\infty,
\end{equation}
so that we may apply Theorem \ref{borelltis cor} to obtain
\begin{equation*}
\begin{aligned}
\; &\Prob \left(\left\{ \| h \|_{\TT^2} > \frac{u}{a(1+O(aS))}
\right \}\right)\\
\; &=O\left(\exp\left[\frac{\alpha_h
u}{a(1+O(aS))}-\frac{u^2}{2a^2\sigma_h^2(1+O(aS))^2}\right]\right)
\end{aligned}
\end{equation*}

To get a lower bound for \eqref{prob:u/a:T2}, we proceed as in
section \ref{sec:BorellTIS} and choose $x_0\in S$ where
$\sigma_h^2=\sup_x r_h(x,x)$ is attained.  Clearly, we shall get a
lower bound in \eqref{prob:u/a:T2} by evaluating $\Prob
\left(\left\{h(x_0)> \frac{u}{a(1+O(aS))}  \right \}\right)$, and
the latter is equal to $\Psi(u/(a\sigma_h^2(1+O(aS))))$.

Next, we remark that $u/(a(1+O(aS)))\sim u/a$ provided $S$ is chosen
so that $aS=o(1)$. Comparing the estimates from above and from
below, we find that
$$
\log \Prob \left(\left\{ \| h \|_{\TT^2} > \frac{u}{a(1+O(aS))}
\right \}\right)=\frac{-u^2}{2a^2\sigma_h^2}.
$$

This concludes the proof of Theorem \ref{thm:supbound-general}
for $\Sigma=\TT^2$, provided \eqref{prob:neglect:BS} holds (that ensures
that the last term will give the dominant contribution to
$\Prob(A_{u,a})$); it remains to show that we can choose $S$ that
will satisfy all the constraints we encountered; accordingly, we
then collect all the inequalities that relate the various parameters
 in the course of the proof, and make sure that a proper choice for
$S$ is possible.

For the applications of Theorem \ref{borelltis cor}, we need both
$S\to\infty$ and \eqref{eq:u/a(1+aS)->inf}; for the latter it is
sufficient to require that $aS=o(1)$ or, equivalently, $S=o(1/a)$
(recall that we assume \eqref{eq:u/a->inf}). To make sure that
\eqref{prob:neglect:BS} holds, we need $u/a=o(S)$. All in all, we
need $u/a = o(S)$ and $S= o(1/a)$ while $S\rightarrow \infty$; the
assumption $u\rightarrow 0$ of the present theorem leaves a handy
margin for a possible choice of $S$, since it implies that $u/a$ is
much smaller than $1/a$.

%%%%%

\vspace{5mm}

\noindent{\bf Step 2b}.

We next consider the case $\Sigma=\mathcal{S}^2$ or
$\sigma=S_\gamma,\gamma\geq 2$.
We want to estimate the probability of the event
$\{||R_1-R_0||_\infty>u\}\cap B_S^c$. Recall from \eqref{R1-R0} that
$$R_1-R_0= R_0(e^{-af}-1)-ae^{-af}h.$$ By the definition of $B_{S}$,
on $B_S^c$, we have for $x\in\Sigma$, $|f(x)|=O(S)$ and
$$
|h(x)|=|\Delta_0 f(x)|= O(S).
$$
Again, we choose $S$ so that
$aS=o(1)$, and it follows easily from the Taylor expansion of
$e^{-af}$ and the definition of $w$ that
\begin{equation}\label{R1-R0:S2}
R_1-R_0=-aw-O(aS)(af+ah)=-aw+O(a^2S^2).
\end{equation}
On $S^2$, the isotropic random field $w$ has constant variance
$\sigma_w^2$ that will be computed later; on $S_\gamma,\gamma\geq 2$
the variance $r_w(x,x)$ is no longer constant, and we denote by
$\sigma_w^2$ its supremum $\sup_{x\in S_\gamma}r_w(x,x)$.

Therefore (cf. \eqref{eq:Prob Aua BSc=Prob ||Df||>u/a(1+eps)})
\begin{equation*}
\begin{split}
\Prob(A_{u,a}\cap B_{S}^{c}) &= \Prob\left(\left\{ \| w+O(aS^2)
\|_{\infty} > \frac{u}{a}   \right\} \cap B^{c}_{S}\right)
\\&= \Prob \left(\left\{ \| w \|_{\infty} > \frac{u}{a}+O(aS^2)
  \right\}\right)+
O(\Prob(B_{S})).
\end{split}
\end{equation*}

Assuming that \eqref{prob:neglect:BS} holds and taking
\eqref{probA:sum} into account, we obtain
\begin{equation}\label{prob:w>u/a:approx}
\begin{aligned}
\Prob (A_{u,a}) &= \Prob \left(\left\{ \| w \|_{\infty} >
\frac{u}{a}+O(aS^2)   \right\}\right)\\
\ &+\  O\left(\exp\left(\alpha_f S-\frac{S^2}{2\sigma_f^2}\right)+
\exp\left(\alpha_h S-\frac{S^2}{2\sigma_h^2}\right)\right).
\end{aligned}
\end{equation}

We choose $S$ so that $\frac{u}{a} = o(S)$ but $S =
o\left(\frac{\sqrt{u}}{a} \right)$, so that this choice is possible
since $\sqrt{u}$ is much larger than $u$, as $u$ is small. We then
have
$$
aS^2=o\left( \frac{u}{a}\right),
$$
so that
\begin{equation*}
\Prob \left(\left\{ \| w \|_{\infty} > \frac{u}{a}+O(aS^2)
\right\}\right)= \Prob \left(\left\{ \| w \|_{\infty} >
\frac{u}{a}(1+o(1)) \right\} \right).
\end{equation*}

As in section \ref{sec:BorellTIS}, we shall estimate the quantity
$\Prob \left(\left\{ \| w \|_{\infty} > \frac{u}{a}(1+o(1)) \right\}
\right)$ from above and below by separate arguments.
We let $$\tau=\tau(u,a,S):=u/a+O(aS^2)=(u/a)(1+o(1)).$$ By Borel-TIS
Theorem \ref{borelltis cor}, there exists $\alpha_w$ such that
\begin{equation}\label{w:prob:upperbd}
\Prob\left(\left\{||w||_\infty>\tau\right\}\right)\leq
\exp\left(\alpha_w\tau-\frac{\tau^2}{2\sigma_w^2}\right).
\end{equation}
This concludes the proof of the upper bound in
\eqref{linfbound:general} in this case.

To get a lower bound in \eqref{linfbound:general}, consider the
point $x_0\in \Sigma$ where $r_w(x,x)$ attains its maximum,
$r_w(x_0,x_0)=\sigma_w^2$. Consider the event $\{|w(x_0)|>\tau\}$.
We find that trivially
\begin{equation}\label{w:prob:lowerbd}
\Prob\left(\left\{||w||_\infty>\tau\right\}\right) \geq{\rm
Prob}(\{|w(x_0)|>\tau\})\geq
\left(\frac{C_1}{\tau}-\frac{C_2}{\tau^3}\right)
\exp\left(-\frac{\tau^2}{2\sigma_w^2}\right).
\end{equation}

We next pass to the limit $u\to 0,u/a\to\infty$; then $\tau\cdot
a/u\to 1$.  Taking logarithm in \eqref{w:prob:upperbd} and
\eqref{w:prob:lowerbd} and comparing the upper and lower bound, we
establish \eqref{linfbound:general} for surfaces of genus $\geq 2$.
This concludes the proof of Theorem \ref{thm:supbound-general}.

\end{proof}

%%%%%%%%%%%%%%%%

\section{Dimension $n>2$}\label{sec:d>2:var}

Let $(M,g_0)$ be a compact orientable $n$-dimensional Riemannian
manifold, $n>2$.  Let $R_0\in C^0(M)$ be the scalar curvature of
$g_0$; we assume that $R_0$ has constant sign.

Let $g_1=e^{af}g_0$ with $f$ as in \eqref{rand:conf} be a conformal
change of metric. The key difference between dimension $2$ and
dimension $n>2$ in our calculations is the presence of the (non-Gaussian)
gradient
term $a^2(n-1)(n-2)|\nabla_0 f|^2/4$ in the equation \eqref{eq:R1
after conf}. We shall assume that $c_j=O(\lambda_j^{-s}),s>n/2+1$.
Then $R_1\in C^0(M)$ a.s. by Proposition \ref{R1:smooth}.

Below, we shall consider the random field
$v(x)=(\Delta_0f)(x)/R_0(x)$. As usual, we let
\begin{equation}\label{supv:n>2:def}
\sigma_v^2=\sup_{x\in M}r_v(x,x).
\end{equation}

Let $P_2(a)$ be the probability of the scalar curvature sign change
after the conformal metric transformation $g_1=e^{af}g_0$, i.e.
$$
P_2(a):=\Prob\{\exists x\in M:\sgn R_1(x)\neq\sgn R_0(x)\}.
$$

%%%%%%%%

\subsection{Negative $R_0$}\label{R0<0:n>2}
We shall first consider the case of $\forall x\in M.\, R_0(x)<0$.

\begin{prop}
\label{prop:R0<0:n>2} Let $(M,g_0)$ be a compact orientable $n$-dimensional
Riemannian manifold, $n>2$, such that the scalar curvature $R_0\in
C^0(M)$ and $\forall x\in M.\, R_0(x)<0$. Assume that
$c_j=O(\lambda_j^{-s}),s>n/2+1$, so that $h,R_1\in C^0(M)$. Then there
exists $\alpha>0$ so that
$$
P_2(a)=O\left(\exp\left(\frac{\alpha}{a(n-1)}-
\frac{1}{2a^2(n-1)^2\sigma_v^2}\right)\right).
$$
\end{prop}

\noindent{\bf Proof of Proposition \ref{prop:R0<0:n>2}.}

Recall that the curvature transformation corresponding to the
conformal change $g_{1}=e^{af}g_{0}$ is given by \eqref{eq:R1 after
conf}, and observe that $R_1e^{af}$ is not Gaussian because of the presence of
the non-Gaussian term $|\nabla_0f|^2$. Then, in order to prove the
assertion of the present proposition, we will get rid of the term
$|\nabla_0f|^2$ in
\eqref{eq:R1 after conf} by taking advantage of its positivity, so
that
$$
\{R_0- a(n-1)\Delta_0 f >0\} \supseteq \{R_0- a(n-1)\Delta_0 f -
a^2(n-1)(n-2)|\nabla_0f|^2/4>0\} .
$$
Therefore
$$
P_2(a)\leq \Prob \{\exists x\in M.\: R_0(x)- a(n-1)(\Delta_0 f)(x)
>0 \}.
$$
Recall that $h=\Delta_0f$. We remark that $\sgn
(R_0(x)-a(n-1)h(x))=-\sgn(1-(n-1)h(x)/R_0(x))$.
Accordingly,
$$
P_2(a)\leq\Prob \{\exists x\in M.\: 1- a(n-1)h/R_0 <0\}=\Prob
\{||h/R_0||_M>1/(a(n-1))\}.
$$
It then remains to apply Theorem \ref{borelltis cor} for
$u=1/(a(n-1))$.

\qed

%%%%%%%

\subsection{Positive $R_0$}\label{R0>0:n>2}

We next consider the more involved case $R_0>0$.  The regularity
assumptions are the same as in section \ref{R0<0:n>2}.

In this section, we shall consider the random field $v=h/R_0$ (considered
earlier in section \ref{R0<0:n>2}).  We shall also consider the quantity
\begin{equation}\label{sigma_w:def:n>2}
\sigma_2=\sup_{x\in M} \frac{\E [|\nabla_0 f(x)|^2]}{R_0(x)}.
\end{equation}

\begin{prop}
\label{prop:R0>0:n>2} Let $(M,g_0)$ be a compact orientable $n$-dimensional
Riemannian manifold, $n>2$, such that the scalar curvature $R_0\in
C^0(M)$ and $\forall x\in M.\, R_0(x)>0$. Assume that
$c_j=O(\lambda_j^{-s}),s>n/2+1$, so that $h,R_1\in C^0(M)$. Then there
exists $\beta>0$ so that
$$
P_2(a)=O\left(\exp\left(\frac{\beta}{a}-\frac{B}{a^2}\right)\right),
$$
where
$$
B=\frac{2+\kappa-\sqrt{\kappa^2+4\kappa}}{\sigma_2n(n-1)(n-2)}.
$$
and
$$
\kappa=\frac{4\sigma_v^2(n-1)}{\sigma_2n(n-2)}.
$$
\end{prop}

\noindent{\bf Proof of Proposition \ref{prop:R0>0:n>2}.}

Note that $\sgn R_1$ is equal to
$$
\sgn\left(1- \frac{a(n-1)h}{R_0} -
\frac{a^2(n-1)(n-2)|\nabla_0f|^2}{4R_0}\right);
$$
here we have used the assumption that $\forall x\in M.\, R_0(x)>0$.
Recall that $P_2(a)$ denotes the probability that $\exists x\in
M:R_1(x)<0$.  We define a random field $u$ to be
$$
u:=\frac{(n-1)h}{R_0} + \frac{a(n-1)(n-2)|\nabla_0f|^2}{4R_0}.
$$

Then $R_1<0$ is equivalent to
$$
\left\| u \right\|_M >\frac{1}{a}.
$$

For every $0\leq\delta\leq 1$, we have
$$
\begin{aligned}
\left\{\|u\|_M\geq \frac{1}{a}\right\} &\subseteq
\left\{\left\|\frac{(n-1)h}{R_0}\right\|_M\geq
\frac{\delta}{a}\right\}\\
\; &\cup \left\{\left\| \frac{a(n-1)(n-2)|\nabla_0f|^2}{4R_0}
\right\|_M\geq \frac{1-\delta}{a}\right\},
\end{aligned}
$$
so that
\begin{equation}\label{prob:ineq:n>2}
\begin{aligned}
\Prob\left\{\|u\|_M\geq \frac{1}{a}\right\} &\leq \Prob
\left\{\left\|\frac{(n-1)h}{R_0}\right\|_M\geq
\frac{\delta}{a}\right\}\\
\; &\; + \Prob \left\{\left\| \frac{a(n-1)(n-2)|\nabla_0f|^2}{4R_0}
\right\|_M\geq \frac{1-\delta}{a}\right\}.
\end{aligned}
\end{equation}

The probability $\Prob\left\{\left\| \frac{(n-1)h}{R_0}\right\|_M \geq
\frac{\delta}{a}\right\}$ can be estimated in a straightforward way using
Theorem \ref{borelltis cor}. Indeed, define the random field
$v:=h/R_0$ (as in section \ref{R0<0:n>2}); as before, let $\sigma_v^2$
be defined by \eqref{supv:n>2:def}.  Then $\left\{\left\|
\frac{(n-1)h}{R_0}\right\|_M \geq \frac{\delta}{a}\right\}$ is equivalent
to
$$
\|v\|_M>\frac{\delta}{a(n-1)},
$$
and the latter can be bounded by
Theorem \ref{borelltis cor} (letting $u=\delta/(a(n-1))$) as
\begin{equation}\label{bound1:n>2}
\Prob \left\{\left\|\frac{(n-1)h}{R_0}\right\|_M\geq
\frac{\delta}{a}\right\} \leq
\exp\left(\frac{\beta_1\delta}{a}-\frac{\delta^2}{2(a(n-1))^2\sigma_v^2}
\right),
\end{equation}
for some constant $\beta_1>0$.

To bound
\begin{equation*}
\begin{aligned}
\; &\ \Prob \left\{\left\| \frac{a(n-1)(n-2)|\nabla_0f|^2}{4R_0}
\right\|_M\geq \frac{1-\delta}{a}\right\}\\
\; &=
\Prob \left\{\left\| \frac{|\nabla_0f|^2}{R_0}
\right\|_M\geq \frac{4(1-\delta)}{a^2(n-1)(n-2)}\right\}
\end{aligned}
\end{equation*}
we need to work harder, as the random field $|\nabla_0 f(x)|^2/R_0(x)$ is
not  Gaussian. The key observation is that we may represent this
field as {\em locally} Gaussian subordinated. Namely, let $$\{U_i:
\;i=1,\dots, m\}$$ be a finite covering of $M$ so that there exists
a geodesic frame $\{E^{i}_1,\dots, E^{i}_n\}$ defined on $U_{i}$. On
$U_{i}$ we have
\begin{equation*}
\frac{|\nabla_0 f(x)|^2}{R_0(x)} = \sum\limits_{k=1}^{n}
\frac{(E_{k}^{i}f(x))^2}{R_0(x)},
\end{equation*}
and we observe that $G_{i,k}(x):=(E_{k}^{i}f(x))/\sqrt{R_0(x)}$
are {\em centered Gaussian}  random fields defined on $U_{i}$.
For each $i,k$ and
$x\in M$ we have
\begin{equation}
\label{eq:varGik<=sigma2:new} \E [G_{i,k}(x)^2] \le
\E\left[\frac{|\nabla f(x)|^2}{R_0(x)} \right] \le \sigma_2
\end{equation}
by the definition \eqref{sigma_w:def:n>2}.

We then have
\begin{equation*}
\left\| \frac{|\nabla_0 f(x)|^2}{R_0(x)}\right\|_{M} = \max\limits_{i} \left\|
\sum\limits_{k=1}^{n} G_{i,k}(x)^2 \right\|_{U_{i}},
\end{equation*}
so that
\begin{equation}\label{eq:Prob nabla^2M<=mProb nabla^Ui:new}
\begin{aligned}
\;&\; \Prob\left\{\left\| \frac{|\nabla_0 f(x)|^2}{R_0(x)}
\right\|_{M} \ge
\frac{4(1-\delta)}{a^2(n-1)(n-2)}   \right\}\\
\; &\le \sum\limits_{i=1}^{m} \Prob\left\{\left\|
\frac{|\nabla_0 f(x)|^2}{R_0(x)}
\right\|_{U_{i}} \ge \frac{4(1-\delta)}{a^2(n-1)(n-2)}
\right\}
\\
\;&\le m\Prob\left\{\left\| \frac{|\nabla_0 f(x)|^2}{R_0(x)}
\right\|_{U_{i_{0}}} \ge
 \frac{4(1-\delta)}{a^2(n-1)(n-2)}  \right\}
\end{aligned}
\end{equation}
where $i_{0}=i_{0}(a)$ maximizes the probability
\begin{equation*}
\Prob\left\{\left\| \frac{|\nabla_0 f(x)|^2}{R_0(x)}
\right\|_{U_{i}} \ge
 \frac{4(1-\delta)}{a^2(n-1)(n-2)}  \right\}
\end{equation*}
for $1\le i\le m$.

Therefore we need to bound
\begin{equation}
\label{eq:Prob|grad|^2<=sum Prob Ei^2:new}
\begin{aligned}
\;&\ \Prob\left\{\left\| \frac{|\nabla_0 f(x)|^2}{R_0(x)}
\right\|_{U_{i_{0}}} \ge
\frac{4(1-\delta)}{a^2(n-1)(n-2)}  \right\}\\
\; &= \Prob\left\{\left\| \sum\limits_{k=1}^{n}
\frac{(E_{k}^{i_{0}}f(x))^2}{R_0(x)}
\right\|_{U_{i_{0}}} \ge
\frac{4(1-\delta)}{a^2(n-1)(n-2)}  \right\}\\
\; &\le \sum\limits_{k=1}^{n} \Prob \left\{\left\|
\frac{|E_{k}^{i_{0}}f(x)|}{\sqrt{R_0(x)}} \right\|_{U_{i_{0}}} \ge
\sqrt{\frac{4(1-\delta)}{a^2n(n-1)(n-2)}}  \right\}.
\end{aligned}
\end{equation}
We may bound each of the summands using the Borel-TIS inequality as
\begin{equation*}
\begin{aligned}
\; &\ \Prob \left\{\left\|
\frac{|(E_{k}^{i_{0}}f(x))|}{\sqrt{R_0(x)}}
\right\|_{U_{i_{0}}} \ge
\sqrt{\frac{4(1-\delta)}{a^2n(n-1)(n-2)}}  \right\} \\
\; &\le \exp\left( \frac{\beta_2}{a}-
\frac{2(1-\delta)}{a^2\sigma_2n(n-1)(n-2)}    \right),
\end{aligned}
\end{equation*}
where we exploited \eqref{eq:varGik<=sigma2:new}; the constant
$\beta_{2}$  absorbs the $2$ factor coming from the possibility that
we might have either a positive or negative sign. Plugging the last
estimate into \eqref{eq:Prob|grad|^2<=sum Prob Ei^2:new} and the
resulting bound into \eqref{eq:Prob nabla^2M<=mProb nabla^Ui:new}, we
finally obtain, possibly choosing a larger constant $\beta_{2}$ to
absorb the constants in front of the exponent,
\begin{equation}
\label{eq:Pr()nabla f^2>= bnd:new}
\begin{aligned}
\; &\;\Prob\left\{\left\| \frac{|\nabla_0 f(x)|^2}{R_0(x)}
\right\|_{M} \ge
\frac{4(1-\delta)}{a^2(n-1)(n-2)}  \right\}\\
\; &\le \exp\left( \frac{\beta_2}{a}-
\frac{2(1-\delta)}{a^2\sigma_2n(n-1)(n-2)} \right).
\end{aligned}
\end{equation}

We next choose $\delta$ in an optimal way, so that the negative
exponents in \eqref{eq:Pr()nabla f^2>= bnd:new} and
\eqref{bound1:n>2} match, i.e. so that
\begin{equation}\label{delta:quadratic}
\frac{\delta^2}{2(n-1)^2\sigma_v^2}=\frac{2(1-\delta)}{\sigma_2n(n-1)(n-2)}.
\end{equation}
or, letting $\kappa=\frac{4\sigma_v^2(n-1)}{\sigma_2n(n-2)}$,
$$
\delta^2+\kappa\delta-\kappa=0.
$$
It is easy to check that the root
$\delta_0=(\sqrt{\kappa^2+4\kappa}-\kappa)/2$ satisfies the required
inequality $0<\delta<1$ and thus gives an admissible solution to
\eqref{delta:quadratic}.  Substituting $\delta_0$, we find that the
exponents in \eqref{delta:quadratic} are both equal to
$$
B=\frac{2+\kappa-\sqrt{\kappa^2+4\kappa}}{\sigma_2n(n-1)(n-2)}.
$$
Substituting into \eqref{eq:Pr()nabla f^2>= bnd:new} and
\eqref{bound1:n>2} finishes the proof of Proposition
\ref{prop:R0>0:n>2}.

\qed

%%%%%%%%%%%%

\section{$Q$-curvature}\label{sec:Qcurv}
The $Q$-curvature was first studied by Branson and later by Gover,
Orsted, Fefferman, Graham, Zworski, Chang, Yang, Djadli, Malchiodi
and others.  We refer to \cite{BG} for a detailed survey.

\subsection{Conformally covariant operators} Here we summarize some
useful results in \cite{BG}.  Let $M$ be a manifold of dimension
$n\geq 3$.  Let $m$ be even, and
$m\notin\{n+2,n+4,\ldots\}\Leftrightarrow m-n \notin 2\zed^+$. Then
there exists on $M$ an elliptic operator $P_m$ (GJMS operators of
Graham-Jenne-Mason-Sparling, cf. \cite{GJMS}).

We shall restrict ourselves to even $n$, and to $m=n$.  We shall
denote the corresponding operator $P_n$ simply by $P$. It satisfies
the following properties: $P=\Delta^{n/2} + lower\ order\ terms$.
$P$ is formally self-adjoint (Graham-Zworski \cite{GZ},
Fefferman-Graham \cite{FG}). Under a conformal change of metric
$\tilde{g}=e^{2\omega}g$, the operator $P$ changes as follows:
$\widetilde{P}=e^{-n\omega}P$. $P$ has a polynomial expression in
(Levi-Civita connection) $\nabla$ and (scalar curvature) $R$, with
coefficients that are rational in dimension $n$.

The operator $P_4=\Delta_g^2+\delta[(2/3)R_g g-2{\rm Ric}_g]d$ is
called the {\em Paneitz operator}.

\subsection{$Q$-curvature and its key properties}

We shall henceforth only consider manifolds of even dimension $n$.
The $Q$-curvature in dimension $4$ was defined by Paneitz as
follows:
\begin{equation}\label{defQ:dim4}
Q_g=-\frac{1}{12}\left(\Delta_g R_g-R_g^2+3|{\rm Ric}_g|^2\right).
\end{equation}

In higher dimensions, $Q$-curvature is a local scalar invariant
associated to the operator $P_n$.  It was introduced by T. Branson
in \cite{Branson}; alternative constructions were provided in
\cite{FG,FH} using the {\em ambient metric} construction.

$Q$-curvature is equal to $1/(2(n-1))\Delta^{n/2}R$ modulo nonlinear
terms in curvature. Under a conformal change of variables
$\tilde{g}=e^{2\omega}g$ on $M^n$, the $Q$-curvature transforms as
follows \cite[(4)]{BG}:
\begin{equation}\label{transf:Qcurv}
P\omega+Q=\tilde{Q}e^{n\omega}.
\end{equation}
Integral of the $Q$-curvature is conformally invariant.

A natural problem is the existence of metrics with constant
$Q$-curvature in a given conformal class. In the following
proposition, we summarize results due to Chang and Yang, and Djadli
and Malchiodi in dimension $4$, and to Ndiaye in arbitrary even
dimension $n>4$ \cite{CY,DM,N}.
\begin{prop}\label{prop:constQ}
Let $(M,g)$ be a compact Riemannian manifold of even dimension
$n\geq 4$, and assume that  $M$ satisfies the following ``generic''
assumptions:
\begin{itemize}
\item[i)] In dimension $n=4$, the assumptions are (\cite{DM}): $\ker
P_n = \{const\}$, and $\int_M Q dV\neq 8\pi^2 k,k=1,2,\ldots$
\item[ii)] In even dimension $n>4$, the assumptions are (\cite{N}): $\ker
P_n = \{const\}$, and $\int_M Q dV\neq (n-1)!\omega_n
k,k=1,2,\ldots,$ where $(n-1)!\omega_n=\int_{S^n} Q dV$, the
integral of $Q$-curvature for the round $\mathcal{S}^n$.
\end{itemize}
Then there exists a metric $g_Q$ on $M$ in the conformal class of
$g$ with constant $Q$-curvature. If $n=4,\int_M Q dV < 8\pi^2$,
$P_4\geq 0$ and $\ker P_4=\{const\}$, then $g_Q$ is unique,
\cite[Thm 2.2]{CY}.
\end{prop}
If  $g$ has positive scalar curvature and $M\neq \mathcal{S}^4$,
then the assumption $\int_M Q dV < 8\pi^2$ is satisfied; if in
addition $\int_M Q\geq 0$, then the assumptions $P_4\geq 0$ and
$\ker P_4=\{const\}$ are also satisfied.

%%%%%%%%

\subsection{Generalizing the results for scalar curvature}
We explain the strategy to generalize our results for scalar
curvature to $Q$-curvature. We consider a manifold $M$ with a
``reference'' metric $g_0$ such that $Q$-curvature has constant sign
and a conformal perturbation $g_1=e^{2af}g_0$ where $a$ is a
positive number; we expand $f$ in a series of eigenfunctions of $P$.
Next, we use formula \eqref{transf:Qcurv} to study the induced
curvature $Q_1$. Finally, we use methods of Adler-Taylor to prove
sharper estimates for the probability for homogeneous manifolds with
constant $Q$-curvature.

We remark that in every conformal class where the generic conditions
of \cite{DM,N} hold, there exist metrics with $Q$-curvature of
constant sign.

%%%%%%%%%%

\subsection{$Q$-curvature in a conformal class}
Let $M$ be a manifold of even dimension $n$, and let $g_0$ be a
metric with $Q$-curvature $Q_0$.

In the Fourier expansions considered below, we shall restrict our
summation to {\em nonzero} eigenvalues of $P_n$.  We remark that the
assumptions of Proposition \ref{prop:constQ}, then $\ker
P_n=\{const\}$.

Let $P=P_n$ have $k$ negative eigenvalues (counted with
multiplicity); denote the corresponding spectrum by
$P\psi_j=-\mu_j\psi_j$, for $1\leq j\leq k$, where
$0>-\mu_1\geq-\mu_2\geq \ldots \geq -\mu_k$.  The other nonzero
eigenvalues are positive, and the corresponding spectrum is denoted
by $P\phi_j=\lambda_j\phi_j$, for $j\geq 1$, where
$0<\lambda_1\leq\lambda_2\leq\ldots$.

Consider the change of metric $g_1=e^{2af}g_0,$ where we let
\begin{equation}\label{series:P}
f=\sum_{i=1}^k b_i\psi_i +\sum_{j=1}^\infty a_j\phi_j,
\end{equation}
and where $b_i\sim \mathcal{N}(0,t_i^2)$ and $a_j\sim
\mathcal{N}(0,c_j^2)$.

We define $h:=-Pf$, and substituting into \eqref{transf:Qcurv}, we
find that
\begin{equation}\label{Q1 curvature}
Q_1e^{naf}=Q_0 -ah =Q_0+a \left(\sum_{j=1}^\infty
\tilde{a}_j\phi_j -\sum_{i=1}^k \tilde{b}_i\psi_i\right) ,
\end{equation}
where $\tilde{a}_j\sim \mathcal{N}(0,\lambda_j^2c_j^2)$ and
$\tilde{b}_i\sim \mathcal{N}(0,t_i^2\mu_i^2)$.

\begin{remark}
It follows that $Q_1e^{naf}(x)$ is Gaussian with expectation
$Q_0(x)$ and covariance function
\begin{equation}\label{h:covar:Q}
a^2\cdot r_h(x,y)=a^2\left(\sum_{i=1}^k
t_i^2\mu_i^2\psi_i(x)\psi_i(y) +\sum_{j=1}^\infty
\lambda_j^2c_j^2\phi_j(x)\phi_j(y)\right).
\end{equation}
\end{remark}

%%%%%%%%

\subsection{Regularity}

It is easy to see that the regularity of the random field in
\eqref{series:P} is determined by the principal symbol
$\Delta^{n/2}$ of the GJMS operator $P=P_n$.  The following
Proposition is then a straightforward extension of Proposition
\ref{prop:sobolev_regularity}:

\begin{prop}\label{regularity:Q}
Let $f$ be defined as in \eqref{series:P}.
If $c_j=O(\lambda_j^{-t})$ and $t> 1+\frac{k}{n}$, then
$f\in C^k$. Similarly, if $c_j=O(\lambda_j^{-t})$ and $t>2+\frac{k}{n}$ then
$Pf\in C^k$.
\end{prop}

%%%%%%%%

\subsection{Using Borel-TIS to estimate the probability that
$Q$-curvature changes sign}\label{Q:sign}

Consider a metric $g_0$ where $Q_0(x)$ has constant sign. We remark
that such metric always exists in the conformal class of $g_0$ if
Proposition \ref{prop:constQ} holds.

Let $f$ be as in  equation \eqref{series:P} and such that $Pf$ is
a.s. $C^0$. We remark that it follows from Proposition
\ref{regularity:Q} that this happens if $c_j=O(\lambda_j^{-t})$
where $t>2$.

Let $g_1=e^{2af} g_0$. Denote the $Q$-curvature of $g_1$ by  $Q_1$;
then it follows from \eqref{transf:Qcurv} that
\begin{equation}\label{Q1:sign}
\sgn(Q_1)=\sgn(Q_0)\sgn(1-ah/Q_0)
\end{equation}
It follows that $Q_1$ changes sign iff $\sup_{x\in M}
h(x)/Q_0(x)>1/a$.

We denote by $v(x)$ the random field $h(x)/Q_0$. It follows from
\eqref{h:covar:Q} that the covariance function of $v(x)$ is equal to
\begin{equation}\label{v:covar:Q}
r_v(x,y)=\frac{1}{Q_0(x)Q_0(y)}\left(\sum_{i=1}^k
t_i^2\mu_i^2\psi_i(x)\psi_i(y) +\sum_{j=1}^\infty
\lambda_j^2c_j^2\phi_j(x)\phi_j(y)\right).
\end{equation}
We let
\begin{equation}\label{supvar:v:Q}
\sigma^2_v:=\sup_{x\in M} r_v(x,x).
\end{equation}

As for the scalar curvature, we make the following
\begin{definition}\label{P2(a):Q}
Denote by $P_2(a)$ the probability that the $Q$-curvature $Q_1$ of
the metric $g_1=g_1(a)$ changes sign.
\end{definition}

\begin{theorem}\label{thm:probQ:sign}
Assume that $Q_0\in C^0(M)$ and that $c_j=O(\lambda_j^{-t}),t>2$.
Then there exist constants $C_1>0$ and $C_2$ such that the
probability $P_2(a)$ satisfies
$$
(C_1 a) e^{-1/(2a^2\sigma_v^2)}\leq P_2(a)\leq
e^{C_2/a-1/(2a^2\sigma_v^2)},
$$
as $a\to 0$.  In particular
$$
\lim_{a\to 0} a^2\ln P_2(a)=\frac{-1}{2\sigma_v^2}.
$$
\end{theorem}

\noindent{\bf Proof of Theorem \ref{thm:probQ:sign}.} It follows
from the assumptions of the theorem and from Proposition
\ref{regularity:Q} that $v\in C^0(M)$ a.s., and hence the Borell-TIS
theorem applies. The rest of the proof follows the proof of Theorem
\ref{prop:negative1}.

\qed

%%%%%%%%%%

\subsection{$L^\infty$ bounds for the
$Q$-curvature}\label{sec:Q:Linfinity}

Here we extend the results in section \ref{sec:Linfinity} to
$Q$-curvature. We have not pursued similar questions for the scalar
curvature in dimension $n\geq 3$ due to the presence of the gradient
term in the transformation formula \eqref{eq:R1 after conf}. For
the $Q$-curvature, there is no gradient term in the corresponding
transformation formula \eqref{transf:Qcurv}, which allows us to
establish the following
\begin{theorem}
Let $(M,g_0)$ be an $n$-dimensional compact orientable Riemannian
manifold, with $n$ even. Assume that $Q_0\in C^0(M)$, and that
$c_j=O(\lambda_j^{-t}),\;  t>2$, so that by Proposition
\ref{regularity:Q} the random fields $f$ and $h$ are a.s. $C^2$. Let
$w:=h-nQ_0f$, denote by $r_w(x,y)$ its covariance function and set
$$\sigma_w^2:=\sup_{x\in M} r_w(x,x).$$ Let $a\rightarrow 0$ and
$u\rightarrow 0$ so that
\begin{equation*}
\label{Q:u/a->inf} \frac{u}{a}\rightarrow \infty.
\end{equation*}
Then
\begin{equation*}\label{Q:linfbound:general}
\log \Prob(\| Q_1-Q_0  \|_{\infty} > u ) \sim
-\frac{u^2}{2a^2\sigma_w^2}.
\end{equation*}
\end{theorem}

\begin{proof}
The proof of this theorem is very similar to the one presented in
step (2c) of Theorem \ref{thm:supbound-general}. We will have to
deal with the fact that neither $f$, $h$ or $w$ have constant
variance. We start by defining  the ``bad" event $B_S$, for $S>0$,
\begin{equation*}
B_{S}=\{||f||_\infty>S\}\cup\{ ||h||_\infty>S\}.
\end{equation*}
By Theorem \ref{borelltis cor},
 there exist two constants $\alpha_f$ and
$\alpha_h$ such that

\begin{equation*}
{\rm Prob} (B_S) = O\left(\exp\left(\alpha_f
S-\frac{S^2}{2\sigma_f^2}\right)+ \exp\left(\alpha_h
S-\frac{S^2}{2\sigma_h^2}\right)\right).
\end{equation*}
for $\sigma_f^2:=\sup_{x\in M} r_f(x,x)$ and
$\sigma_h^2:=\sup_{x\in M} r_h(x,x)$, where $r_f$ and $r_h$ are the
covariance functions of $f$ and $h$ respectively. \ \\ \ As before,
we denote $A_{u,a}$ the event $\{|| Q_{1}-Q_0 ||_{\infty}> u\}$ and
observe that  ${\rm Prob} (A_{u,a})={\rm Prob} (A_{u,a}\cap
B_S)+{\rm Prob} (A_{u,a}\cap B_S^c).$ We estimate ${\rm Prob}
(A_{u,a}\cap B_S)$ trivially : ${\rm Prob}(A_{u,a}\cap B_S)\leq {\rm
Prob} (B_S).$ This implies that
\begin{equation*}\label{prob:neglect:A in Q-curv}
{\rm Prob} (A_{u,a}\cap B_S)= O\left(\exp\left(\alpha_f
S-\frac{S^2}{2\sigma_f^2}\right)+ \exp\left(\alpha_h
S-\frac{S^2}{2\sigma_h^2}\right)\right).
\end{equation*}

In order to estimate ${\rm Prob} (A_{u,a}\cap B_S^c)$, recall that
in the case of Q-curvature, it follows from \eqref{transf:Qcurv}
that
\begin{equation*}
Q_1-Q_0=Q_0(e^{-naf}-1)+ahe^{-naf}.
\end{equation*}
By the definition of $B_{S}$, on $B_S^c$, we have for $x\in M$ that
$$|f(x)|=O(S) \mbox{\;\;\; and\;\;\;} |h(x)|= O(S).$$

We choose $S$ so that
$aS=o(1)$. It follows easily from the Taylor expansion of $e^{-af}$
and the definition of $w$ that
\begin{equation*}
Q_1-Q_0=ah-anQ_0f+O(a^2S^2)=aw+O(a^2S^2).
\end{equation*}

Therefore,
\begin{equation*}
\begin{split}
\Prob(A_{u,a}\cap B_{S}^{c}) &= \Prob\left(\left\{ \| w+O(aS^2)
\|_{\infty}
> \frac{u}{a}   \right\} \cap B^{c}_{S}\right)
\\&= \Prob \left(\left\{ \|w\|_{\infty}
> \frac{u}{a}+O(aS^2)\right\}\right)+ O(\Prob(B_{S})).
\end{split}
\end{equation*}

The notation was conveniently chosen so that the rest of the proof
be identical to that of Theorem \ref{thm:supbound-general} step
(2c).

\end{proof}

%%%%%%%%%%%%

\section{Conclusion}
In the present paper, we considered a random conformal perturbation
$g_1$ of the reference metric $g_0$ (which we assumed to have
constant scalar curvature $R_0$), and studied the following
questions about the scalar curvature $R_1$ of the new metric:
\begin{itemize}
\item[i)] Assuming $R_0\neq 0$, estimate the probability that $R_1$
changes sign;
\item[ii)] Estimate the probability that $||R_1-R_0||_\infty>u$,
where $u>0$ is a parameter.
\end{itemize}
We also studied analogous questions for Branson's $Q$-curvature.

The measures on metrics in a conformal class that we define
``localize'' to a reference metric $g_0$ in the limit $a\to 0$,
where $g_0$ is the Yamabe metric when we study the scalar curvature;
or the metric with $Q_0\equiv const$ when we study the
$Q$-curvature.

In the proofs, we used Theorem \ref{thm:AT E[chi]} due to
Adler-Taylor, or (Borel-TIS) Theorem \ref{borelltis cor}, which
requires lower regularity of random fields, but gives less precise
estimates.

\subsection{Further questions}
There are numerous questions that were not addressed in the present
paper. We concentrated on he study of {\em local} geometry of spaces
of positively- or negatively-curved metrics (see Remark
\ref{remark:local-geometry}), but it seems extremely interesting to
study {\em global} geometry of these spaces,
\cite{GL,Kat,Lo,Ros06,Sch87,SY79-1,SY82,SY87}.

Another interesting question that seems tractable concerns the study
of the {\em nodal set} of $R_1$ i.e. its zero set. That set, like
the sign of $R_1$, only depends on the quantity $R_0-a(n-1)\Delta_0
f-a^2(n-1)(n-2)|\nabla_0 f|^2/4$ (or $R_0-a\Delta_0 f$ in dimension
two). It also seems interesting to study other characteristics of
the curvature (whether it changes sign or not), such as its $L^p$
norms, the structure of its nodal domains (if it changes sign), and
of its sub- and super-level sets.

Also, it seems quite interesting to study related questions for
Ricci and sectional curvatures in dimension $n\geq 3$.

Another important question concerns an appropriate definition of
measures on the space of Riemannian metrics not restricted to a
single conformal class.

A very important question concerns the study of metrics of lower
regularity than in the present paper, appearing e.g. in
$2$-dimensional quantum gravity, cf. \cite{DS}.

In addition, it seems very interesting to study various questions
about random metrics that are influenced by curvature, such as
various geometric invariants (girth, diameter, isoperimetric
constants etc); spectral invariants (small eigenvalues of $\Delta$,
determinants of Laplacians, estimates for the heat kernel,
statistical properties of eigenvalues and of the spectral function,
etc); as well as various questions related to the geodesic flow or
the frame flow on $M$, such as existence of conjugate points,
ergodicity, Lyapunov exponents and entropy, etc.

We plan to address these and other questions in subsequent papers.

%%%%%%%%%%

\appendix

\section{Metrics with positive and negative scalar
curvature}\label{sec:metr-posneg} In this section we review some
results about the spaces of metrics of positive and negative scalar
curvature. In dimension two, $S^2$ admits the metric of positive
curvature, and surfaces of genus $\geq 2$ admit metrics of negative
curvature. For connected manifolds $M$ of dimension $n\geq 3$,
Kazdan and Warner proved the following ``trichotomy'' theorem:
\begin{itemize}
\item[i)] If $M$ admits a metric of nonnegative and not identically
$0$ scalar curvature, then any $f\in C^\infty(M)$ can be realized as
a scalar curvature of some Riemannian metric.
\item[ii)] If $M$ is not in (i) and admits a metric of vanishing
scalar curvature, then $f\in C^\infty(M)$ can be realized as a
scalar curvature provided $f(x)<0$ for some $x\in M$, or else
$f\equiv 0$.
\item[iii)]  If $M$ is not in (i) or (ii), then $f\in C^\infty(M)$
can be realized as a scalar curvature provided $f(x)<0$ for some
$x\in M$.
\end{itemize}

%%%%%%%%%%%%%%

\subsection{Negative scalar curvature} Denote by $S^-(M)$ the space of
metrics of negative scalar curvature on a manifold $M$ of dimension
$n\geq 3$; it follows from results of Aubin and Kazdan-Warner that
$S^-(M)$ is always nonempty. A fundamental theorem about the
structure of $S^-(M)$ was proved by J. Lohkamp \cite{Lo}, who showed
that $S^-(M)$ is connected and aspherical (and hence is
contractible).  He also showed that the space $S_{-1}(M)$ of metrics
of constant curvature $-1$ is contractible.  It is shown in
\cite{Kat} that on a Haken manifold, the moduli space
$S_{-1}(M)/Diff_0(M)$ (where $Diff_0(M)$ denotes the the group of
diffeomorphisms isotopic to the identity) is also contractible,
similarly to Teichmuller spaces for surfaces of genus $\geq 2$ in
dimension $2$.

In a different paper, Lockhamp showed that $S^-(M)$ and $S_{-1}(M)$
are never convex.

%%%%%%

\subsection{Positive scalar curvature} Questions about existence and
spaces of metrics of positive scalar curvature are more complicated
than similar questions for negative scalar curvature. Here we recall
some of the less technical results in recent Rosenberg's survey
\cite{Ros06}.  We make no attempt to give a complete survey, we just
want to list some examples of manifolds where the results of our
paper hold.

There are several techniques for proving results about non-existence
of metrics of positive scalar curvature on a given manifold. We
assume that $M$ is compact, closed, oriented manifold.
\begin{itemize}
\item[i)] For spin manifolds with positive scalar curvature, it
follows from the work of Lichnerowicz that all harmonic spinors
(lying in the kernel of the Dirac operator) have to vanish; it
follows from the work of Lichnerowicz and Hitchin that any manifold
with nonvanishing Hirzebruch genus $\hat{A}(M)$ has no metrics of
positive scalar curvature.  We refer to \cite{Ros86} and
\cite{Ros06} for further non-existence results that use index theory
of Dirac operator, and for relations to Novikov conjectures.
\item[ii)] It follows from the work of Schoen and Yau on minimal surfaces
\cite{SY79-1,SY79-2,SY82} that if $N$ is a stable
$(n-1)$-dimensional submanifold of an $n$-dimensional manifold $M$
with positive scalar curvature, and if $N$ dual to a nonzero element
in $H^1(M,\zed)$, then $N$ also admits a metric of positive scalar
curvature.  It was shown in \cite{SY79-2} that if on a $3$-manifold
$\pi_1(M)$ contains a product of two cyclic groups, or a subgroup
isomorphic to the fundamental group of a compact Riemann surface of
genus $>1$, then M cannot have a metric of positive scalar
curvature.  Moreover, it was shown in \cite{SY87} that a closed
aspherical $4$-manifold cannot admit a metric of positive scalar
curvature.
\item[iii)] Furhter negative results for $4$-manifolds can be obtained using
Seiberg-Witten theory. It was shown by Witten and Morgan that on a
$4$-manifold with $b_2^+(M)>1$, if the Seiberg-Witten invariant
$SW(\xi)\neq 0$ for some $spin^c$ structure $\xi$, then M does not
admit a metric of positive scalar curvature. Taubes showed that
existence of a symplectic structure on a $4$-manifold with
$b_2^+(M)>1$ implies the previous condition. We refer to
\cite{Ros06} for a summary of results in case $b_2^+(M)=1$.
\end{itemize}

In the positive direction, it was shown by Gromov-Lawson and
Schoen-Yau \cite{GL,SY79-1} that if $M_0$ is a manifold (not
necessarily connected) of positive scalar curvature, then any
manifold $M_1$ obtained from $M_0$ by a surgery in codimension $\geq
3$ also admits a metric of positive scalar curvature. In dimension
$n\geq 5$, the condition $w_2(M)\neq 0$ (where $w_2(M)$ is the
second Stiefel-Whitney class of $M$) implies the existence of
metrics with positive scalar curvature.

%%%%%%%

\subsection{Moduli spaces of metrics of positive scalar curvature}
Denote by $S^+(M)$ the space of metrics of negative scalar curvature
on a manifold $M$ of dimension $n\geq 3$ (the space $S^+(S^2)$ is
contractible).  In general, $S^+(M)$ is not connected.  For example,
Hitchin \cite{Hit} showed that on a $n$-dimensional spin manifold
$M$ admitting a metric of positive scalar curvature,
$\pi_0(S^+(M))\neq 0$ if $n\equiv 0$ or $1$ $\mod 8$, and
$\pi_1(S^+(M))\neq 0$ if $n\equiv 0$ or $-1$ $\mod 8$.  For more
general results, we refer to the results of Stolz \cite[Thm.
2.3]{Ros06}.  Gromov and Lawson proved that $S^+(S^7)$ has
infinitely many components.  The same result holds for
$M=S^{4k-1},k>2$, cf. \cite{Ros06}.   we refer to \cite[Thm. 2.7,
2.8]{Ros06} for further results. In dimension $4$, Ruberman showed
that there exists a simply-connected $M^4$ with infinitely many
metrics of positive scalar curvature that are {\em concordant} (i.e.
restrictions to $s=0$ and $s=1$ of a metric of positive scalar
curvature on $M\times [0,1]$), but not isotopic.

%%%%%%%%%%%%%%

\section{Yamabe Problem}\label{sec:yamabe}

Yamabe problem concerns finding metrics with constant scalar
curvature in a conformal class of metrics on a manifold of dimension
$n>3$.  The problem was formulated by H. Yamabe \cite{Yam} (who also
claimed to solve it, but the proof contained gaps), and solved by
Trudinger, Aubin and Schoen \cite{Tr, Au76, Sch84}.

The corresponding smooth metric $g_\gamma$ (called {\em Yamabe
metric}) exists in every conformal class, and minimizes the total
scalar curvature $$({\rm vol (M,g)})^{-(n-2)/n}\int_M R(g) dV(g),$$
when $g$ is restricted to a conformal class of metrics.  The {\em
sign} of $R(g_\gamma)$ is uniquely determined by the conformal
class. The conformal class is called {\em positive} (resp. {\em
negative}) if $R(g_\gamma)>0$ (resp. $<0$); {\em non-positive} and
{\em non-negative} conformal classes are defined similarly.

The Yamabe metric is unique in every non-positive conformal class,
\cite[\S 1]{And05}. The space $\mathcal{Y}^-(M)$ of all negative
unit volume Yamabe metrics on $M$ forms a smooth, infinite
dimensional manifold, transverse to the space of conformal classes,
in the space of all unit volume metrics on $M$, \cite{Bes,Sch87}.
Here the spaces of $\mathcal{Y}^-(M)$ (resp. $\mathcal{Y}^+(M),
\mathcal{Y}^0(M)$) of metrics of negative (resp. positive, zero)
scalar curvature are discussed in Appendix \ref{sec:metr-posneg}.

The scalar curvature defines a smooth function
$R:\mathcal{Y}^-\to\reals$, whose critical points are Einstein
metrics on $M$ with negative scalar curvature.  Similar results hold
for non-positive conformal classes, \cite[\S 1]{And05}.

The situation is much more complicated for positive conformal
classes: there the Yamabe metrics are not unique in general, e.g. on
$S^n$, the group of M\"obius transformations acts on the space of
Yamabe metrics. However, they are unique {\em generically} (for an
open dense set of metrics in the space of positive conformal
classes), cf. \cite[Thm. 1.1]{And05}.

%%%%%%%%%%%%%%

\section{Validity of applying Adler-Taylor for $h=\Delta_{0}f$ on $\sphere$}

\label{apx:h attain, E[chi]=Prob}

In this appendix we justify the application of two results. In
section \ref{apx:suit reg S2} we prove that $h$ satisfies the
conditions of Theorem \ref{thm:AT E[chi]} due to Adler-Taylor,
namely that $h$ is attainable. In section \ref{apx:Echi f>=u = P
f>=0} we prove that the sufficient conditions for \eqref{eq:Prob
||h||>=u = E[chi]} hold (i.e. the hypotheses of ~\cite{AT08},
Theorem 14.3.3).

\subsection{Attainability of $h=\Delta_0 f$ on $\sphere$.}

\label{apx:suit reg S2}

The goal of the present section is to prove that the random field $h=\Delta f$
on the $2$-dimensional sphere,
given by \eqref{eq:h=Deltaf def}, is attainable
(cf. Definition \ref{def:attainable}), assuming the coefficients $c_{m}$
are decaying sufficiently rapidly as
\begin{equation}
\label{eq:cm << m^-s, s>7}
c_m=O(m^{-s}) \text{ for } s>7
\end{equation}
(see \eqref{eq:cn <<>>n^-s} and the assumptions of Theorem \ref{thm:prob R1>0 S2}).

First, Lemma \ref{R1sphere:smooth} and the assumptions on the decay
of $c_{m}$ imply that $h$ is $C^2$ a.s. In fact, from the
$C^{k,\beta}$ version of the Sobolev embedding theorem (cf.
\cite[Thm. 2.10, 2nd part]{Au98}), it follows from the strict
inequality $s>7$ that there exists $\beta>0$ such that
\begin{equation}
\label{eq:h in C^2,beta}
h\in C^{2,\beta}(S^2) \text{ a.s.};
\end{equation}
this will be used later. Next we check conditions \eqref{it:hi, hij
nondeg} and \eqref{it:rij reg} of Definition \ref{def:attainable} of
attainability.

For each $y\in \sphere$ let
$\delta:T\rightarrow\sphere$ be the spherical coordinates with pole at $y$,
where $T=[0,\pi]\times [0,2\pi]$. Namely, we let $(\theta_{y},\phi_{y})$ be the
standard spherical coordinates of $y$ and define
\begin{equation*}
\delta_{y}(\theta,\phi) = ( \sin(\theta-\theta_{y})\cos(\phi-\phi_{y}),
\sin(\theta-\theta_{y})\sin(\phi-\phi_{y}) , \cos(\theta-\theta_{y})),
\end{equation*}
and $\psi_{y} := \delta_{y}^{-1}$. Let $x\in\sphere$ be a point and
$(U_{x},\psi_{y})$ be any small chart with $x\in U_{x}$,
$\psi_{y}(U_{x})\subseteq \R^2$ for some $y\in\sphere$. We claim
that choosing $y$ appropriately, a sufficiently small chart $U$
satisfies condition \eqref{it:hi, hij nondeg} of attainability. This
is, of course, sufficient to form a finite atlas, by the compactness
of the sphere.

First, at any point $t\in U_{x}$, the random vector
$$H(t)=(h_{i}(t),h_{ij}(t))_{i<j}\in\R^{5}$$ is mean zero Gaussian
(here the derivatives are w.r.t. the cartesian coordinates in
$\R^{2}$). Therefore we have to check that its covariance matrix
$C_{H(t)}\in M_{5}(\R)$ is non-degenerate; by the locality it is
sufficient to check that $C_{H(x)}$ is nonsingular. The matrix
$C_{H(x)}$ depends, in general, on the choice of $y$; we are free to
choose $y$ as we wish.

It turns out that for $y$ for which $\phi = \frac{\pi}{2}$,
$C_{H(x)}$ is of a  particularly simple form. For this choice of
$y$, we compute $C_{H(x)} = \sum\limits_{m\ge 1} c_{m}C_{m}$ (with
finite entries), where the single eigenspace covariance matrices
$C_{m}$ are given explicitly by
\begin{equation*}
C_{m} = \left( \begin{matrix} \frac{E_{m}}{2}I_{2} & 0_{2\times 3}
\\ 0_{3\times 2}  & \Omega_{3\times 3}^m \end{matrix} \right),
\end{equation*}
with
\begin{equation*}
\Omega_{3\times 3}^m = \frac{E_{m}}{8} \left( \begin{matrix}
3E_{m}-2 &E_{m}+2
\\ E_{m}+2 &3E_{m}-2 \\ & & E_{m}-2   \end{matrix}   \right),
\end{equation*}
and it is then easy to use \eqref{eq:cm << m^-s, s>7} in order to
check that the entries of $C_{H(x)}$ are finite and the matrix is
nonsingular.

\begin{remark}
A priori, it seems that non-degeneracy of $H$ in one point is
sufficient, thanks to the isotropic property of $h$. However, one
should bear in mind that introducing a chart breaks the symmetry, so
that the second derivatives are no longer isotropic, being dependent
on the local properties of the corresponding frame. This is unlike
the first (directional) derivatives, which depend only on the
direction of the frame at the given point.
\end{remark}

As for condition \eqref{it:rij reg} of Definition \ref{def:attainable},
it follows easily from \eqref{eq:h in C^2,beta}, the latter implying
$$r_{h_{ij}}(\cdot,  t), r_{h_{ij}}(t,  \cdot) \in C^{0,\beta}(\sphere)$$
for every $t \in \sphere$.

\subsection{Relation of the expected Euler characteristic of the excursion set
and the excursion probability}

\label{apx:Echi f>=u = P f>=0}

The goal of the present section is to justify the application
of ~\cite{AT08}, Theorem 14.3.3 on $h=\Delta f$
given by \eqref{eq:h=Deltaf def}. Recall that the covariance
function of $h$ is $r_{h}$, given by \eqref{eq:sph cov fnc}.

In addition to the assumptions already validated in
the previous section we are required to show that
\begin{equation}
\label{eq:r(x,y)=1 <=> x=y}
r_{h}(x,y) = 1 \;\; \Leftrightarrow \;\; x=y
\end{equation}
(recall that for every $x\in\sphere$ we have $r(x,x) = 1$
by the assumption \eqref{eq:cn var 1}). This condition rules
out degeneracies such as periodic processes.

We claim that \eqref{eq:r(x,y)=1 <=> x=y} holds if and only if
there exists an {\em odd} $m_{0}$ so that
\begin{equation}
\label{eq:c_m odd > 0}
c_{m_{0}} > 0.
\end{equation}
That is guaranteed by one of the assumptions in Theorem
\ref{thm:prob R1>0 S2}.

To see that we note (see e.g. ~\cite{W}) that for every $m\ge 1$,
$|P_{m}(t)| \le 1$ for $t\in [-1,1]$, $P_{m}(1) = 1$;
\begin{equation*}
|P_{m}(t)| = 1 \Leftrightarrow t=\pm 1,
\end{equation*}
and $P_{m}$ is even or odd, for $m$ even or odd respectively.
Thus we have by \eqref{eq:sph cov fnc}
\begin{equation*}
|r_{h}(x,y)| \le \sum\limits_{m=1}^{\infty} c_{m} = 1
\end{equation*}
by \eqref{eq:cn var 1}, and the equality may hold only if
$\cos(d(x,y)) = \pm 1$, i.e. $x=\pm y$. In case $x=-y$ this
may not hold by \eqref{eq:c_m odd > 0}.

\end{document}